\documentclass[a4paper, 12pt, twoside,openright]{article}

\usepackage{enumitem}
\usepackage[section]{placeins}	
\usepackage{epigraph}
\usepackage[font={small}]{caption}
\usepackage{appendix}

\usepackage{etoolbox}
\patchcmd{\thebibliography}
  {\settowidth}
  {\setlength{\itemsep}{-2pt plus 0.1pt}\settowidth}
  {}{}
\apptocmd{\thebibliography}
  {\small}
  {}{}

\usepackage{amsmath,amssymb,mathrsfs}
\usepackage{amsfonts}			
\usepackage{nicefrac}			

\usepackage{bbm} 

\usepackage{stmaryrd}
\usepackage{calc,ifthen,xspace}

\usepackage[all,cmtip]{xy}
\usepackage{array}
\usepackage{multirow}
\usepackage{booktabs}
\usepackage{colortbl}
\usepackage{tabularx}
\usepackage{multirow}
\usepackage{threeparttable}

\usepackage{graphicx}			
\usepackage{subcaption}
\usepackage{pdfpages}
\usepackage{rotating}
	
\usepackage{tikz}

\usepackage{tikz-cd}
\usetikzlibrary{knots}
\usetikzlibrary{decorations.markings}
	
\usepackage{changepage}

\usepackage{xspace}
\usepackage{textcomp}
\usepackage{array}
\usepackage{hyphenat}

\newcolumntype{C}{>$c<$}

\usepackage[pdftex,pdfborder={0 0 0},
			colorlinks=true,
			linkcolor=blue,
			citecolor=red,
			pagebackref=true,
			]{hyperref}	

\usepackage[top=2.5cm, bottom=2cm, left=3cm, right=2.5cm]{geometry}
\usepackage{cleveref}
\crefformat{section}{\S#2#1#3} 
\crefformat{subsection}{\S#2#1#3}
\crefformat{subsubsection}{\S#2#1#3}

\AtBeginDocument{}

\renewcommand{\leq}{\ensuremath{\leqslant}} 
\renewcommand{\geq}{\ensuremath{\geqslant}}

\let\originalleft\left 
\let\originalright\right
\renewcommand{\left}{\mathopen{}\mathclose\bgroup\originalleft}
\renewcommand{\right}{\aftergroup\egroup\originalright}

\DeclareMathOperator{\Lie}{\mathcal L}
\DeclareMathOperator{\Der}{Der}
\DeclareMathOperator{\ad}{ad}

\DeclareMathOperator{\ima}{Im}
\DeclareMathOperator{\coker}{coker}

\DeclareMathOperator{\Aut}{Aut}

\DeclareMathOperator{\Inn}{Inn}

\DeclareMathOperator{\rk}{rk}
\DeclareMathOperator{\Tors}{Tors}

\DeclareMathOperator{\Sym}{\mathfrak S}
\DeclareMathOperator{\Homeo}{Homeo}
\DeclareMathOperator{\PHomeo}{PHomeo}

\newcommand{\Z}{\mathbb Z}

\newcommand{\D}{\mathbb D}

\usepackage{amsthm}

\makeatletter
\newtheorem*{rep@theorem}{\rep@title}
\newcommand{\newreptheorem}[2]{%
\newenvironment{rep#1}[1]{%
 \def\rep@title{#2 \ref{##1}}%
 \begin{rep@theorem}}%
 {\end{rep@theorem}}}
\makeatother

\theoremstyle{plain}
\newtheorem{theo}{Theorem}[section]
\newreptheorem{theo}{Theorem}
\newtheorem{fact}[theo]{Fact} 
\newtheorem{lem}[theo]{Lemma}
\newtheorem{prop}[theo]{Proposition}
\newreptheorem{prop}{Proposition}
\newtheorem{propdef}[theo]{Proposition-definition}
\newtheorem{cor}[theo]{Corollary}
\newreptheorem{cor}{Corollary}

\newreptheorem{pb}{Problem}

\numberwithin{equation}{theo}

\theoremstyle{definition}
\newtheorem{defi}[theo]{Definition}
\newtheorem{conv}[theo]{Convention}
\newtheorem{nota}[theo]{Notation}

\newtheorem{rmq}[theo]{Remark}

\let\oldpagenumbering\pagenumbering
\renewcommand{\pagenumbering}[1]{%
	\cleardoublepage
	\oldpagenumbering{#1}
}

\author{Jacques \scshape{Darn\'e}}
\title{Braids, inner automorphisms \newline and the Andreadakis problem}
\date{December 6, 2020}

\hypersetup{
    pdfauthor={Jacques Darn\'e},
    pdfsubject={Inner automorphisms},
    pdftitle={Braids, inner automorphisms  and the Andreadakis problem},
    pdfkeywords={algebraic topology, group theory, automorphisms of free groups, braid groups}
}

\begin{document}

\maketitle

\begin{abstract}
In this paper, we generalize the tools that were introduced in \cite{Darne2} in order to study the Andreadakis problem for subgroups of $IA_n$. In particular, we study the behaviour of the Andreadakis problem when we add inner automorphisms to a subgroup of $IA_n$. We notably use this to show that the Andreadakis equality holds for the pure braid group on $n$ strands modulo its center acting on the free group $F_{n-1}$, that is, for the (pure, based) mapping class group of the $n$-punctured sphere acting on its fundamental group. 
\end{abstract}

\section*{Introduction}
\addcontentsline{toc}{section}{Introduction}

In his 1962 PhD.~thesis \cite{Andreadakis}, Andreadakis studied two filtrations on the group of automorphisms of the free group $F_n$. More precisely, they were filtrations on the subgroup $IA_n$ of $\Aut(F_n)$ consisting of automorphisms acting trivially on the abelianization $F_n^{ab} \cong \Z^n$. These filtrations were:
\begin{itemize}[itemsep=-3pt,topsep=3pt]
\item The \emph{lower central series} $IA_n = \Gamma_1(IA_n) \supseteq \Gamma_2(IA_n) \supseteq \cdots$.
\item The filtration $IA_n = \mathcal A_1 \supseteq \mathcal A_2 \supseteq \cdots$ now known as the \emph{Andreadakis filtration}.
\end{itemize}
He showed that there are inclusions $\mathcal A_i \supseteq \Gamma_i(IA_n)$, and he conjectured that these were equalities. This question became known as the \emph{Andreadakis conjecture}, and it turned out to be a very difficult one, which is still nowdays shrouded with mystery. The interest of this question (and its difficulty) lies notably in the fact that the definitions of these filtrations are very different in nature, and thus what we understand about them is too. For instance, it is very easy to test whether a given element lies in some~$\mathcal A_j$, but there is no known efficient procedure for testing whether the same element belongs to $\Gamma_j$ or not. On the other hand, producing elements of $\Gamma_j$ is not difficult, but we do not know any good recipe for producing elements of $\mathcal A_j$ (apart, of course, from the one producing elements of $\Gamma_j$). So far, these difficulties have been overcome only for very small values of $n \geq 3$ and very small degrees to show, using computer calculations, that the conjecture is in fact false \cite{Bartholdi1}.

The same question can be asked for any subgroup $G$ of $IA_n$. Namely, if $G$ is such a subgroup, then we can ask whether the inclusions $\Gamma_i(G) \subseteq G \cap \mathcal A_i$ are equalities. The answer is obviously negative for some subgroups which are embedded in $IA_n$ in a wrong way (take for instance a cyclic subgroup of $\Gamma_2(IA_n)$). But if $G$ is nicely embedded in~$IA_n$, we can hope that these filtrations on $G$ are equal, in which case we say that the subgroup $G$ of $IA_n$ \emph{satisfies the Andreadakis equality}. 

Interesting examples of such subgroups include the pure braid group, embedded in~$IA_n$ \emph{via} Artin's action on the free group, for which the Andreadakis equality was shown in \cite{Darne2}. Another example is the pure welded braid group, also acting on the free group \emph{via} Artin's action. This subgroup is the group of (pure) basis-conjugating automorphisms of $F_n$, also known as the McCool group $P \Sigma_n$. We still do not know whether the Andreadakis equality holds for this one. A version of the latter problem up to homotopy has been considered in \cite{Darne4}. In all these cases, the Andreadakis equality can be seen as a comparison statement between different kinds of invariants of elements of the group (see \cite{Darne4}).

\paragraph*{Automorphisms of free groups.}
Consider the $2$-sphere $\mathbb S^2$ with $n$ marked points, and let us choose a basepoint different from the marked points. The group of isotopy classes of orientation-preserving self-homeomorphisms of $\mathbb S^2$ fixing the base point and each marked point is isomorphic to the quotient $P_n^*$ of the pure braid group on $n$ strands by its center. This group acts canonically on the fundamental group of the sphere with the $n$ marked points removed, which is free on $n-1$ generators. In \cite{Magnus}, this action is shown to be faithful (see also our appendix \cref{section_action_Pn*}). Moreover, the induced action on the homology of the punctured sphere, which is the abelianization of its fundamental group, is trivial. As a consequence, the mapping class group $P_n^*$ identifies with a subgroup of $IA_{n-1}$. Our main goal is the following theorem :

\begin{reptheo}{Andreadakis_for_FnPn}
The subgroup $P_{n+1}^*$ of $IA_n$ satisfies the Andreadakis equality.
\end{reptheo}

It turns out that the subgroup $P_{n+1}^*$ of $IA_n$ is generated by pure braids, together with inner automorphisms of $F_n$. Our strategy of proof relies on this fact. Indeed, our key result (Th.~\ref{key_theo} and Cor.~\ref{cor_to_key_theo}) is 
a theorem allowing us to decide, for a given subgroup $K$ of $IA_n$, whether $K \cdot \Inn(F_n)$ satisfies the Andreadakis equality when $K$ does. The reader should note that such a result allowing us to pass from a subgroup to another one is quite exceptional : the Andreadakis problem (and, more generally, the lower central series) does not usually behave well when passing to smaller or bigger groups. In this regard, although $P_{n+1}^*$ strictly contains $P_n$ and is contained in $P \Sigma_n$, we can hardly see our present result as a step further in the study of the Andreadakis problem for $P \Sigma_n$. It has to be considered as a new interesting example in itself, and also as a good pretext to develop new tools in the study of the Andreadakis problem.

We give two other applications of our key result (Th.~\ref{key_theo} and Cor.~\ref{cor_to_key_theo}). By applying it to the subgroup $K = IA_n^+$ of triangular automorphisms, we get the Andreadakis equality for a somewhat bigger subgroup (Th.~\ref{Andreadakis_for_FnIAn+}). Moreover, by applying it to the subgroup $K = P \Sigma_n^+$ of triangular basis-conjugating automorphisms, we recover the result obtained in \cite{Ibrahim} showing that the Andreadakis equality holds for the group of \emph{partial inner automorphisms} defined and studied in \cite{Bardakov-Neshchadim}.

\paragraph*{Inner automorphisms.}
The question of whether the Andreadakis equality holds can be asked, more generally, for subgroups of $IA_G$, where $G$ is any group, and $IA_G$ is the group of automorphisms of $G$ acting trivially on its abelianization $G^{ab}$. In order to get our results about subgroups of $IA_n$, we need to show that the Andreadakis equality holds for the subgroup $\Inn(F_n)$ of inner automorphisms of $F_n$. Although this result is fairly easy to get (see~\cref{par_Fn}), we choose to develop the general theory of the Andreadakis problem for inner automorphisms, and we show, in particular, the following result (where $\Lie(G)$ denotes the graded Lie ring obtained from the lower central series of $G$):

\begin{repcor}{Andreadakis_for_Inn_LCS}
Let $G$ be a group. The Andreadakis equality holds for $\Inn(G)$ if and only if every central element of $\Lie(G)$ is the class of some central element of $G$.
\end{repcor}

We give several examples and counter-example, the main one being the case of the group $G = P_n$ of pure braids. In order to study it, we give a calculation of the center of $P_n$ which can readily be adapted to a calculation of the center of its Lie ring, and we prove:

\begin{reptheo}{Andreadakis_for_Inn(Pn)}
The subgroup $\Inn(P_n)$ of $IA(P_n)$  ($\subset \Aut(P_n)$) satisfies the Andrea\-dakis equality.
\end{reptheo}

\paragraph{Outline of the paper:}

The first section is devoted to recalling the needed definitions and results from the theory of group filtrations, in particular filtrations on automorphism groups and on braid groups. In \cref{section_inner_aut}, we study the Andreadakis problem for inner automorphisms, which turns out to be very much related to calculations of centers of groups and of their Lie rings. We then turn to the calculation of the center of the pure braid group, of which we give a version that generalizes easily to a calculation of the center of the associated Lie ring; this allows us to solve the Andreadakis problem for inner automorphisms of the pure braid group (\cref{section_centers}). Then, in \cref{section_products}, we prove our key result (Th.~\ref{key_theo}), giving a criterion for deducing the Andreadakis equality for a product of subgroups from the Andreadakis equality for these subgroups. Finally, the last section is devoted to applications to subgroups of automorphisms of free groups, namely triangular automorphisms, triangular basis-conjugating automorphism, and the pure braid group on $n$ strands modulo its center acting on $F_{n-1}$.

In addition to our main results, we put in an appendix a comparison between the Drinfeld-Kohno Lie ring and the Lie ring of so-called \emph{special derivations} of the free Lie ring (which we call \emph{braid-like}), boiling down to some rank calculations. In a second appendix, we write down a new proof of the faithfulness of the action of the braid group on $n$ strands modulo its center on $F_{n-1}$, which involves less calculations than the ones in the literature, and we gather some useful group-theoretic results.

\tableofcontents

\numberwithin{theo}{section}
\section{Reminders}\label{section_reminders}

We recall here some of the basics of the general theory of (strongly central) group filtrations and the Andreadakis problem. Details may be found in \cite{Darne1, Darne2}.

\subsection{Filtrations on groups}

Since the only filtrations we consider in the present paper are \emph{strongly central} ones (in the sense of \cite{Darne1}), we adopt Serre's convention \cite{Serre} and we simply call them  \emph{filtrations}. The systematic study of such filtrations was initiated by Lazard \cite{Lazard}, who called them \emph{$N$-series}.

\begin{nota}
Let $G$ be a group. If $x, y \in G$, we denote by $[x,y]$ their commutator $xyx^{-1}y^{-1}$, and we use the usual exponential notations $x^y = y^{-1}xy$ and ${}^y\! x = yxy^{-1}$ for conjugation in $G$. If $A, B \subseteq G$ are subsets of $G$, we denote by $[A,B]$
 the subgroup generated by commutators $[a,b]$ with $a \in A$ and $b \in B$.
 \end{nota}

\begin{defi}
A \emph{filtration} $G_*$ on a group $G$ is a sequence of nested subgroups $G = G_1 \supseteq G_2 \supseteq G_3  \supseteq \cdots$ satisfying:
\[\forall i,j \geq 1,\ \ [G_i, G_j] \subseteq G_{i+j}.\]
\end{defi}

If $G_*$ and $H_*$ are two filtrations on the same group $G$, we write $G_* \subseteq H_*$ if $G_i \subseteq H_i$ for all $i$. The minimal filtration (for the inclusion relation) on a given group $G$  is its lower central series $\Gamma_*(G)$, defined as usual by $\Gamma_1(G) = G$ and $\Gamma_{i+1}(G) = [G,\Gamma_i(G)]$ when $i \geq 1$. Recall that $G$ is called \emph{nilpotent} (resp.~\emph{residually nilpotent}) if $\Gamma_i(G) = \{1\}$ for some $i$ (resp.~if $\bigcap \Gamma_i(G) = \{1\}$). Since the lower central series is the minimal filtration on $G$, if $G_i = \{1\}$ for some $i$ (resp.~if $\bigcap G_i = \{1\}$) for any filtration $G_*$ on $G = G_1$, then  $G$ is nilpotent (resp.~residually nilpotent).

\begin{conv}
Let $G$ be a group endowed with a filtration $G_*$. Let $g$ be an element of $G$. If there is an integer $d$ such that $g \in G_d - G_{d+1}$, it is obviously unique. We then call $d$ the \emph{degree} of $g$ with respect to $G_*$. The notation $\overline g$ will denote the class of $g$ in some quotient $G_i/G_{i+1}$ ; if the integer $i$ is not specified, it will be assumed that $i = d$, which means that $\overline g$ denotes the only non-trivial class induced by $g$ in some $G_i/G_{i+1}$. If such a $d$ does not exist (that is, if $g \in \bigcap G_i$), we say that $g$ has degree $\infty$ and we put $\overline g = 0$.
\end{conv}

Recall that to a filtration $G_*$ we can associate a graded Lie ring (that is, a graded Lie algebra over $\mathbb Z$):
\begin{propdef}
If $G_*$ is a group filtration, the graded abelian group $\Lie(G_*) := \bigoplus G_i/G_{i+1}$ become a graded Lie ring when endowed with the Lie bracket $[-,-]$ induced by commutators in $G$. Precisely, with the above convention, this bracket is defined by:
\[ \forall x \in G_i,\ \forall y \in G_j,\ [\overline x, \overline y]:= \overline{[x,y]} \in \Lie_{i+j}(G_*),\]
where $\Lie_k(G_*) = G_k/G_{k+1}$ denotes the homogeneous elements of degree $k$ in $\Lie(G_*)$.
\end{propdef}

\begin{conv}
When no filtration is specified on a group $G$, it is implied that $G$ is endowed with its lower central series $\Gamma_*(G)$. In particular, we denote $\Lie(\Gamma_*(G))$ simply by $\Lie(G)$.
\end{conv}

Since products of commutators become sums of brackets inside the Lie algebra, the following fundamental property follows easily from the definition of the lower central series:
\begin{prop}\label{engdeg1} 
The Lie ring $\Lie(G)$ is \emph{generated in degree $1$}. Precisely, it is generated (as a Lie ring) by $\Lie_1(G) = G^{ab}$. As a consequence, if $G$ is of finite type, then each $\Lie_n(G)$ is too.
\end{prop}

Let us recall that the construction of the associated Lie ring from a filtration is a functor from the category of filtrations and filtration-preversing group morphisms to the category of graded Lie rings. This functor $\Lie : G_* \mapsto \Lie(G_*)$ will be referred to as \emph{the Lie functor}. We will use the fact that it is \emph{exact}. Precisely, if $G_*$, $H_*$ and $K_*$ are group filtrations, a \emph{short exact sequence of filtrations} is a short exact sequence $H_1 \hookrightarrow G_1 \twoheadrightarrow K_1$ of groups, such that the morphisms are filtration-preserving, and such that they induce a short exact sequence of groups
$H_i \hookrightarrow G_i \twoheadrightarrow K_i$
for all integer $i \geq 1$. As a consequence of the nine Lemma in the category of groups, the Lie functor sends a short exact sequence of filtrations to a short exact sequence of graded Lie rings \cite[Prop.~1.24]{Darne1}.

\subsection{Actions and Johnson morphisms}

Before introducing the Andreadakis problem, we introduce one of our main tools in its study, which is the \emph{Johnson morphism} associated to an \emph{action} of a filtration on another one. Recall that the categorical notion of an \emph{action} of an object on another one in a protomodular category leads to the following definition:
\begin{defi}
Let $K_*$ and $H_*$ be filtrations on groups $K = K_1$ and $H = H_1$. An \emph{action} of $K_*$ on $H_*$ is a group action of $K$ on $H$ by automorphisms such that:
\[\forall i, j \geq 1,\ \ [K_i, H_j] \subseteq H_{i+j},\]
where commutators are computed in $H \rtimes K$. Precisely:
\[\forall k \in K,\ \forall h \in H,\ \ [k,h] = (k \cdot h) h^{-1},\]
where $k \cdot h$ denotes the image of $h$ by the action of $k$.
\end{defi}

Given a group action of $K$ on $H$, the above conditions are exactly the ones required for the sequence of subgroups $(H_i \rtimes K_i)_{i \geq 1}$ to be a filtration on $H \rtimes K$, denoted by $H_* \rtimes K_*$. Then $\Lie(H_* \rtimes K_*)$ is a semi-direct product of $\Lie(H_*)$ and $\Lie(K_*)$, encoding an action of $\Lie(K_*)$ on $\Lie(H_*)$ by derivations, described explicitly by the formula:
\[\forall k \in K,\ \forall h \in H,\ \ \overline k \cdot \overline h = \overline{(k \cdot h)h^{-1}}.\]
This action can also be seen as a morphism from $\Lie(K_*)$ to the Lie ring $\Der(\Lie(H_*))$ of derivations of $\Lie(H_*)$, called the \emph{Johnson morphism} associated to the action of  $K_*$ on $H_*$:
\[\tau : \left\{\renewcommand{\arraystretch}{1.5}
\begin{array}{clc}
\Lie(K_*) &\longrightarrow &\Der(\Lie(H_*)) \\
\overline k &\longmapsto & \left(\overline h \mapsto \overline{(k \cdot h)h^{-1}}\right).
\end{array}
\right.\]

Actions of filtrations can be obtained from group actions \emph{via} the following:
\begin{propdef}
Let $K$ be a group acting on another group $H$ by automorphisms, and let $H_*$ be a filtration on $H = H_1$. Then, there is a greatest one among filtrations $K_*$ on a subgroup $K_1$ of $K$ such that the action of $K$ on $H$ induces an action of $K_*$ on $H_*$. This filtration is denoted by $\mathcal A_*(K, H_*)$ and is defined by:
\[\mathcal A_j(K, H_*) = \left\{k \in K \ \middle|\ \forall i \geq 1,\ [k, H_i] \subseteq H_{i+j} \right\}.\]
The filtration $\mathcal A_*(\Aut(H), H_*)$, denoted simply by $\mathcal A_*(H_*)$, is called the \emph{Andreadakis filtration associated to $H_*$}. When furthermore $H_* = \Gamma_* H$, we denote it by $\mathcal A_*(H)$, and we call it the \emph{Andreadakis filtration associated to $H$}. The filtration $\mathcal A_*(H)$ is a filtration on $\mathcal A_1(H) =: IA_H$, which is the group of automorphisms of $H$ acting trivially on $H^{ab}$.
\end{propdef}

The following fact was the initial motivation for introducing such filtrations~\cite{Kaloujnine2}:

\begin{fact}\label{A*_stopping}
In the above setting, if $K$ acts faithfully on $H$, and if $H_i = \{1\}$ for some~$i$ (resp.~if $\bigcap H_i = \{1\}$), then $\mathcal A_{i-1}(K, H_*) = \{1\}$ (resp.~$\bigcap \mathcal A_i(K, H_*) = \{1\}$). It particular, $\mathcal A_1(K, H_*)$ must then be nilpotent (resp.~residually nilpotent).
\end{fact}

Remark that if the morphism $a : K \rightarrow \Aut(H)$ represents the action of $K$ on $H$, then  $\mathcal A_*(K,H_*) = a^{-1}(\mathcal A_*(H_*))$.

\subsection{Lower central series of semi-direct products}

We now briefly recall the description of lower central series of semi-direct product, which will be an important ingredient in the proof of our key theorem (Th.~\ref{key_theo}). This description is based on the following general construction:

\begin{propdef}\textup{\cite[Def.~3.3]{Darne2}.}\label{SCD_relative}
Let $G$ be a group, and let $H$ be a normal subgroup of $G$. We define a filtration $\Gamma_*^G (H)$ on $H$ by:
\[\begin{cases} \Gamma_1^G (H):= H, \\ \Gamma_{k+1}^G:= [G, \Gamma_k^G (H)]. \end{cases}\]
\end{propdef}

If a group $K$ acts on another group $H$ by automorphisms, then we denote $\Gamma^{H \rtimes K}_*(H)$ only by  $\Gamma^K_*(H)$. It is the minimal filtration on $H$ which is acted upon by $\Gamma_*(K)$ \emph{via} the action of $K$ on $H$.
\begin{prop}\textup{\cite[Prop.~3.4]{Darne2}.}\label{lcs_of_sdp}
Let $K$ be a group acting on another group $H$ by automorphisms. Then:
\[\forall i \geq 1,\ \Gamma_i(H \rtimes K) = \Gamma^K_i(H) \rtimes \Gamma_i(K).\]
\end{prop}

Moreover, under the right conditions, $\Gamma^K_*(H)$ is in fact equal to $\Gamma_*(H)$:
\begin{prop}\textup{\cite[Prop.~3.5]{Darne2}.}\label{lcs_of_adp}
Let $K$ be a group acting on another group $H$ by automorphisms. The following conditions are equivalent:
\begin{itemize}[itemsep=-3pt,topsep=3pt]
\item The action of $K$ on $H^{ab}$ is trivial.
\item $[K,H] \subseteq [H,H] = \Gamma_2(H)$.
\item The action of $K$ on $H$ induces an action of $\Gamma_*K$ on $\Gamma_*H$.
\item $\Gamma_*^K(H) = \Gamma_*(H)$.
\item $\forall i \geq 1,\ \Gamma_i(H \rtimes K) = \Gamma_i(H) \rtimes \Gamma_i(K)$.
\item $\Lie(H \rtimes K) \cong \Lie(H) \rtimes \Lie(K).$ 
\item $(H \rtimes K)^{ab} \cong H^{ab} \times K^{ab}.$
\end{itemize}
When these conditions are satisfied, we say that the semi-direct product $H \rtimes K$ is an \emph{almost-direct} one.
\end{prop}

\subsection{The Andreadakis problem}

Let $G_*$ be a group filtration. The Andreadakis problem is concerned with the comparison of two filtrations on $\mathcal A_1(G_*)$, namely the Andreadakis filtration $\mathcal A_*(G_*)$ and the filtration $\Gamma_*(\mathcal A_1(G_*))$. The latter is the minimal filtration on the group $\mathcal A_1(G_*)$, hence $\mathcal A_*(G_*)$ contains $\Gamma_*(\mathcal A_1(G_*))$. Now, if $K$ is a subgroup of $\mathcal A_1(G_*)$ (that is, $K$ acts on $G$ by automorphisms preserving the $G_i$, and the induced action on $\Lie(G_*)$ is trivial), then we can restrict these fitrations to $K$, and the filtrations so obtained must contain $\Gamma_*(K)$, which is the minimal filtration on $K$.

\begin{defi}
Let $G_*$ be a group filtration and $K$ be a subgroup of $\mathcal A_1(G_*)$. We say that $K$ \emph{satisfies the Andreadakis equality with respect to $G_*$} if the following inclusions are equalities:
\[\Gamma_*(K)\ \subseteq\ K \cap \Gamma_*(\mathcal A_1(G_*)) \ \subseteq\ K \cap \mathcal A_*(G_*).\]
If $K$ is a subgroup of $IA_G$, we simply say that $K$ \emph{satisfies the Andreadakis equality} when it does with respect to $\Gamma_*(G)$.
\end{defi}

\begin{rmq}
This definition can be generalized to groups acting on $G$ (in a possibly non-faithful way). Precisely, let $K$ act on a group $G$ \emph{via} a morphism $a : K \rightarrow \Aut(G)$, and let $G$ be endowed with a filtration $G_*$. Suppose that $a(K) \subseteq \mathcal A_1(G_*)$ (that is, $K$ acts by automorphisms preserving the filtration, and the induced action on $\Lie(G_*)$ is trivial). Then we get inclusions of filtrations on $K$:
\[\Gamma_*(K)\ \subseteq\ a^{-1}\left(\Gamma_*(\mathcal A_1(G_*))\right) \ \subseteq\ a^{-1}(\mathcal A_*(G_*)) = \mathcal A_*(K, G_*).\]
However, since $\Gamma_*(a(K)) = a(\Gamma_*(K))$, these filtrations are equal if and only if $a(K)$ satisfies the Andreadakis equality with respect to $G_*$, so we can (and will) focus on subgroups of $\Aut(G)$ when studying the difference between such filtrations.
\end{rmq}

The following result is deduced easily from the definitions, and will be our main tool in proving the Andreadakis equality for subgroups of $IA_G$:

\begin{prop} \textup{\cite[Lem.~1.28]{Darne1}.}\label{Johnson_inj}
Let $K_*$ and $H_*$ be group filtrations. The Johnson morphism associated to a given action of $K_*$ on $H_*$ is injective if and only if $K_* = \mathcal A_*(K_1, H_*)$.
\end{prop}

\paragraph{The Andreadakis problem for automorphisms of free groups.} The classical setting is the one when $G = F_n$ is the free group on $n$ generators, and $G_*$ is its lower central series. Then $IA_G$, denoted by $IA_n$, is the subgroup of $\Aut(F_n)$ made of automorphisms acting trivially on $F_n^{ab} \cong \mathbb Z^n$. The Andreadakis filtration associated to $F_n$ is a filtration on $IA_n$ simply denoted by $\mathcal A_*$ and referred to as \emph{the Andreadakis filtration}. Recall that $F_n$ is residually nilpotent, which implies that $\bigcap \mathcal A_i = \{1\}$, thus $IA_n$ is residually nilpotent (Fact \ref{A*_stopping}). Since the Lie algebra of $F_n$ is the free Lie ring $\mathfrak L_n$ on $n$ generators, the Johnson morphism associated to the action of $\mathcal A_*$ on $\Gamma_*(F_n)$ is an injection (Prop.~\ref{Johnson_inj} above) $\tau : \Lie(\mathcal A_*) \hookrightarrow \Der(\mathfrak L_n).$

\subsection{Braids}

We gather here the results we need about Artin's braid groups and filtrations on them. Our main reference here is Birman's book \cite{Birman}. The reader can also consult the original papers of Artin \cite{Artin25, Artin47}.

\subsubsection{Generalities}\label{reminders_braids}

We denote Artin's braid group by $B_n$ and the subgroup of pure braids by $P_n$. Recall that $B_n$ is generated by $\sigma_1, ..., \sigma_{n-1}$ and $P_n$ is generated by the $A_{ij}$ (for $i < j$), which are drawn as follows:

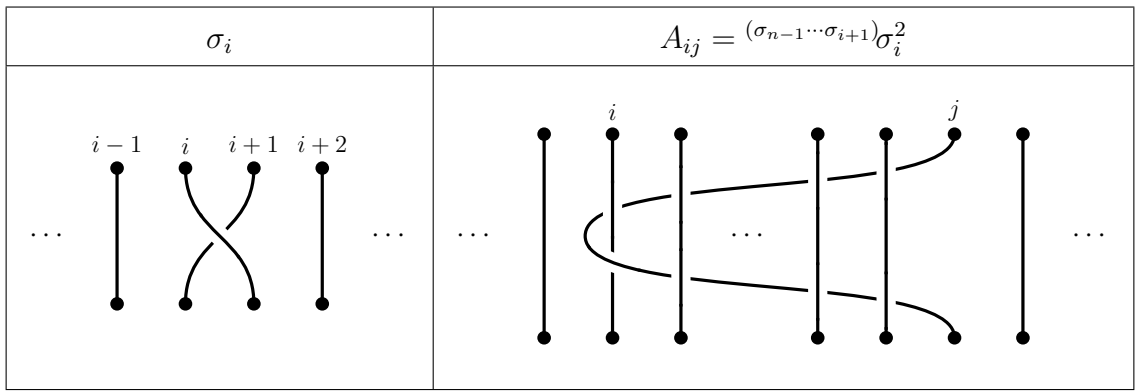
\[
\renewcommand{\arraystretch}{1.5}
\begin{array}{|C|C|}
\hline
$\sigma_i$
&$A_{ij} = {}^{(\sigma_{n-1} \cdots \sigma_{i+1})}\!\sigma_i^2$
\\
\hline
\begin{tikzpicture}[scale=0.9, every node/.style={scale=0.9}]
\path (2, 3) (2, -1);

\node[draw=none] at (-2,1) {$\cdots$};
\node[draw=none] at (3,1) {$\cdots$};

\begin{knot}[clip width = 6, flip crossing = 1]
\strand[very thick]  (-1, 0)  -- (-1, 2)
   node[at end,circle, fill, inner sep=2pt]{}
   node[at start,circle, fill, inner sep=2pt]{}
   node[at end, above=2, scale=0.85]{$i-1$};
   
\strand[very thick]  (0, 0) .. controls  +(0, 1) and +(0, -1) .. (1, 2)
   node[at end,circle, fill, inner sep=2pt]{}
   node[at start,circle, fill, inner sep=2pt]{}
   node[at end, above=2, scale=0.85]{$i+1$};
   
\strand[very thick] (1 ,0) .. controls  +(0, 1) and +(0, -1) .. (0, 2)
   node[at end,circle, fill, inner sep=2pt]{}
   node[at start,circle, fill, inner sep=2pt]{}
   node[at end, above=2, scale=0.85]{$i$};
   
\strand[very thick]  (2, 0)  -- (2, 2)
   node[at end,circle, fill, inner sep=2pt]{}
   node[at start,circle, fill, inner sep=2pt]{}
   node[at end, above=2, scale=0.85]{$i+2$};
\end{knot}
\end{tikzpicture}
&
\begin{tikzpicture}[scale=0.9, every node/.style={scale=0.9}]
\path (2, 4) (2, -.5);
\node[draw=none] at (-2,1.5) {$\cdots$};
\node[draw=none] at (2,1.5) {$\cdots$};
\node[draw=none] at (7,1.5) {$\cdots$};

\begin{knot}[clip width = 6, flip crossing = 1]
\strand[very thick]  (-1, 0)  -- (-1, 3)
   node[at end,circle, fill, inner sep=2pt]{}
   node[at start,circle, fill, inner sep=2pt]{};
\strand[very thick]  (0, 0)  -- (0, 3)
   node[at end,circle, fill, inner sep=2pt]{}
   node[at start,circle, fill, inner sep=2pt]{}
   node[at end, above=2, scale=0.85]{$i$};
\strand[very thick]  (1, 0)  -- (1, 3)
   node[at end,circle, fill, inner sep=2pt]{}
   node[at start,circle, fill, inner sep=2pt]{};
\strand[very thick]  (3, 0)  -- (3, 3)
   node[at end,circle, fill, inner sep=2pt]{}
   node[at start,circle, fill, inner sep=2pt]{};
\strand[very thick]  (4, 0)  -- (4, 3)
   node[at end,circle, fill, inner sep=2pt]{}
   node[at start,circle, fill, inner sep=2pt]{};
\strand[very thick]  
(5, 0)  .. controls  +(0, 1) and +(0, -1) .. (-.4, 1.5)
    node[at start,circle, fill, inner sep=2pt]{}
.. controls  +(0, 1) and +(0, -1) .. (5, 3)
   node[at end,circle, fill, inner sep=2pt]{}
   node[at end, above=2, scale=0.85]{$j$};
\strand[very thick]  (6, 0)  -- (6, 3)
   node[at end,circle, fill, inner sep=2pt]{}
   node[at start,circle, fill, inner sep=2pt]{};
\end{knot}
\end{tikzpicture}  \\ \hline
\end{array}
\]

\begin{conv}
It is often convenient to have $A_{ij}$ defined for all $i,j$ (not only for~$i < j$), using the formulas $A_{ji} = A_{ij}$ and $A_{ii} = 1$.
\end{conv}

\begin{conv}
Drawing braids from top to bottom corresponds to taking products from left to right. Otherwise said, the product $\alpha\beta$ denotes the braid obtained by putting $\alpha$ on top of $\beta$.
\end{conv}

Forgetting the $(n+1)$-th strand induces a projection $P_{n+1} \twoheadrightarrow P_n$ which is split (a section is given by adding a strand away from the other ones). The kernel of this projection identifies with the fundamental group of the plane with $n$ punctures, which is the free group $F_n$. We thus get a decomposition into a semi-direct product:
\[P_{n+1} \cong F_n \rtimes P_n,\]
which encodes an action of $P_n$ on $F_n$ \emph{via} automorphisms. A basis of $F_n$ is given by the elements $x_i := A_{i, n+1}$. A classical result of Artin \cite[Cor.~1.8.3 and Th.~1.9]{Birman} says that this action is faithful, and its image is exactly the group $\Aut^{\partial}_C(F_n)$ of automorphisms preserving the conjugacy class of each generator and fixing the following \emph{boundary element}:

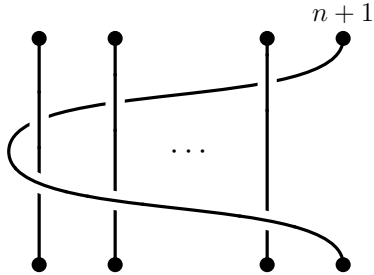
\begin{align*}
\begin{tikzpicture}
\node[draw=none] at (2,1.5) {$\cdots$};
\begin{knot}[clip width = 6, flip crossing/.list={1,3,5}]
\strand[very thick]  (0, 0)  -- (0, 3)
   node[at end,circle, fill, inner sep=2pt]{}
   node[at start,circle, fill, inner sep=2pt]{};
\strand[very thick]  (1, 0)  -- (1, 3)
   node[at end,circle, fill, inner sep=2pt]{}
   node[at start,circle, fill, inner sep=2pt]{};
\strand[very thick]  (3, 0)  -- (3, 3)
   node[at end,circle, fill, inner sep=2pt]{}
   node[at start,circle, fill, inner sep=2pt]{};
\strand[very thick]  
(4, 0)  .. controls  +(0, 1) and +(0, -1) .. (-.4, 1.5)
    node[at start,circle, fill, inner sep=2pt]{}
.. controls  +(0, 1) and +(0, -1) .. (4, 3)
   node[at end,circle, fill, inner sep=2pt]{}
   node[at end, above=2, scale=0.85]{$n+1$};
\end{knot}
\end{tikzpicture} \\[0.7em]
\partial_n := x_1 \cdots x_n = A_{1,n+1} \cdots A_{n,n+1}.
\end{align*}
Since $P_n \cong \Aut^{\partial}_C(F_n)$, such automorphisms are also called \emph{braid automorphisms}. We will often identify $P_n$ with $\Aut^{\partial}_C(F_n)$ in the sequel. 

Remark that their are many possible choices of generators of $P_n$, but the choices that we have made are coherent : they allow us to interpret the above split projection $P_{n+1} \twoheadrightarrow P_{n}$ (defined by forgetting the $(n+1)$-th strand) as the projection from $\Aut^{\partial}_C(F_{n+1})$ onto $\Aut^{\partial}_C(F_n)$ induced by $x_{n+1} \mapsto 1$ (precisely, any braid automorphism $\beta$ sends $x_{n+1}$ to one of its conjugates, thus preserves the normal closure of $x_{n+1}$ and induces an automorphism of $F_{n+1}/x_{n+1} \cong F_n$). The section of this projection defined by adding a strand corresponds to extending canonically automorphisms of $F_n$ to automorphisms of $F_{n+1}$ fixing $x_{n+1}$.

\begin{rmq}
The free subgroup of $B_{n+1}$ generated by the $A_{i, n+1}$ is in fact normalized not only by $P_n$ but also by $B_n$. The corresponding action of $B_n$ on $F_n$ (induced by conjugation in $B_{n+1}$) is faithful, and its image is the group of automorphisms of $F_n$ permuting the conjugacy classes of the $x_i$ and fixing the boundary element.
\end{rmq}

Recall that $B_n$ and $P_n$ can be seen as fundamental groups of configuration spaces. Namely, if $\D$ is the closed disc and $F_n(\D^{\mathrm o}) := \{(x_1, ..., x_n) \in \D^{\mathrm o} \ |\ \forall i \neq j,\ x_i \neq x_j\}$ is the usual configuration space, $P_n$ (resp.~$B_n$) is the fundamental group of $F_n(\D^{\mathrm o})$ (resp.~of its quotient $F_n(\D^{\mathrm o})/\Sym_n$ by the symmetric group). 

They can also be interpreted as Mapping Class Groups. Indeed, if we fix a base configuration in the open disk, evaluation at this base configuration gives a continuous map from the space $\Homeo_\partial(\D)$ of self-homeomorphisms of the disk fixing the boundary pointwise to $F_n(\D^{\mathrm o})$. This map is a locally trivial fibration, whose fiber is the subspace $\PHomeo_\partial(\D, n)$ of homeomorphisms fixing $n$ points of $\D^{\mathrm o}$. Using Alexander's trick, one can use the long exact sequence in homotopy to get an isomorphism:
\[P_n \cong \pi_0(\PHomeo_\partial(\D, n)).\]
This extends to an isomorphism $B_n \cong \pi_0(\Homeo_\partial(\D, n))$ identifying $B_n$ with isotopy classes of self-homeomorphisms of the disk fixing the boundary pointwise and permuting $n$ given points of $\D^{\mathrm o}$.

With this point of view, the Artin action on the free group is induced by the canonical action of homeomorphisms permuting $n$ points on the fundamental group of the disk with these $n$ points removed, which is free on $n$ generators.

\subsubsection{The Drinfeld-Kohno Lie ring}\label{reminders_DK}

A braid acts on $F_n^{ab} \cong \mathbb Z^n$ \emph{via} the associated permutation of the basis. As a consequence, the pure braid group $P_n$ acts trivially on $F_n^{ab}$, hence it is a subgroup of $IA_n$. Since $IA_n$  is residually nilpotent, so is $P_n$. The Lie ring associated to the lower central series of $P_n$ was first determined rationnally in \cite{Kohno}, and it was shown not to have torsion in \cite{Falk-Randell}, where its ranks where computed. Details about the complete description over~$\Z$ that we now recall may be found in the appendix of \cite{Darne2}.

\begin{defi}
The \emph{Drinfeld-Kohno Lie ring} is $DK_n := \Lie(P_n)$.
\end{defi}

Since $P_n$ acts trivially on $F_n^{ab}$, the semi-direct product $P_{n+1} \cong F_n \rtimes P_n$ is an almost-direct one, thus it induces a decomposition of the associated Lie rings:
\[\Lie(P_{n+1}) \cong \mathfrak L_n \rtimes \Lie(P_n),\]
where $\mathfrak L_n = \Lie(F_n)$ is the free Lie ring on $n$ generators $X_i := t_{i, n+1} = \overline A_{i, n+1}$ (for $i \geq n$). Thus,  $\Lie(P_n)$ decomposes as an iterated semi-direct product of free Lie rings. From that, it is not difficult to get a presentation of this Lie ring:

\begin{prop} \textup{\cite[Prop~A.3]{Darne2}.}\label{Drinfeld-Kohno}
The Drinfeld-Kohno Lie ring $DK_n = \Lie(P_n)$ is generated by the $t_{ij}  = \overline A_{ij}$ ($1 \leq i , j \leq n$), under the relations:
\[\begin{cases}
t_{ij} = t_{ji},\ t_{ii} = 0 &\forall i,j,\\ 
[t_{ij}, t_{ik} + t_{kj}] = 0 &\forall i, j, k,\\
[t_{ij}, t_{kl}] = 0 &\text{if } \{i,j\} \cap \{k,l\} = \varnothing.
\end{cases}\]
\end{prop}

Recall that the decomposition $DK_{n+1} \cong \mathfrak L_n \rtimes DK_n$ is encoded in the corresponding Johnson morphism $\tau : DK_n \rightarrow \Der(\mathfrak L_n)$. The following result allows us to identify $DK_n$ with a Lie subring of the Lie ring of derivations of the free Lie ring:

\begin{theo}\textup{\cite[Th.~6.2]{Darne2}.}\label{Andreadakis_for_Pn}
The subgroup $P_n$ of $IA_n$ satistfies the Andreadakis equality. This means exactly that the corresponding Johnson morphism $\tau : DK_n \rightarrow \Der(\mathfrak L_n)$ is injective.
\end{theo}  

\begin{rmq}\label{rmq_on_Milnor_inv}
This statement can be seen as a statement about Milnor $\mu$-invariants. Namely, up to a slight change of viewpoint (corresponding to taking Magnus expansions), $\tau\left(\overline\beta\right)$ corresponds exactly to the set of non-vanishing Milnor invariants of the (pure) braid $\beta$ of minimal degree. In particular, $\tau\left(\overline\beta\right)$ is of degree at least $d$ if and only if the Milnor invariants of degree at most $d-1$ of $\beta$ vanish. With this point of view, the injectivity of $\tau$ means exactly that Milnor invariants of degree at most $d-1$ distinguish braids up to elements of $\Gamma_d(P_n)$ (see also \cite{Mostovoy, Habegger}). Moreover, since $P_n$ (which is an almost-direct product of free groups) is residually nilpotent, this also means that Milnor invariants distinguish braids.
\end{rmq}

In the sequel, we identify $DK_n$ with $\tau(DK_n)$. It is not difficult to see that $DK_n$ is in fact a Lie subring of the Lie ring $\Der_t^{\partial}(\mathfrak L_n)$ of \emph{tangential derivations} (derivations sending each $X_i$ to $[X_i,w_i]$ for some $w_i \in \mathfrak L_n$) vanishing on the \emph{boundary element}, defined by:
\[\overline \partial_n = X_1 + \cdots + X_n = t_{1, n+1} + \cdots + t_{n, n+1} \in \mathfrak L_n \subset DK_{n+1}.\]

\begin{defi}\label{defi_braid_der}
A derivation of the free Lie ring $\mathfrak L_n$ is called \emph{braid-like} (or \emph{special}) if it is an element of $\Der_t^{\partial}(\mathfrak L_n)$. It is called a \emph{braid derivation} if it stands inside $DK_n$.
\end{defi}  

\begin{rmq}[On terminology]
Ihara \cite{Ihara} called \emph{special} the derivation we call \emph{tangential}. The ones we call \emph{braid-like}, he called \emph{normalized special}. More recently, authors working on the Kashiwara-Vergne problem (see for instance \cite{Alekseev}), who were using the word ``tangential" for the former, kept the word ``special" only for Ihara's \emph{normalized special} derivations (dropping the adjective ``normalized"). In the present paper, we use the word ``tangential", but we prefer ``braid-like" over ``special", because we feel that it conveys more meaning.
\end{rmq}

Not all braid-like derivations are braid derivations. The difference between the two Lie subrings of $\Der(\mathfrak L_n)$ is investigated in the appendix (\cref{section_braid_derivations}).

Remark that the projection of $DK_n$ onto $DK_{n-1}$ giving the above decomposition $DK_n \cong \mathfrak L_{n-1} \rtimes DK_{n-1}$ can be seen as the restriction of the projection of $\Der_t^{\partial}(\mathfrak L_n)$ onto $\Der_t^{\partial}(\mathfrak L_{n-1})$ induced by $X_n \mapsto 0$.

\section{Inner automorphisms}\label{section_inner_aut}

This section is devoted to the study of the Andreadakis problem for inner automorphisms of a group $G$, with respect to any filtration $G_*$ on the group. We prove a general criterion (Th.~\ref{Andreadakis_for_Inn}), involving a comparison between the center of $G$ and the center of the Lie ring associated to $G_*$. Since most of our applications will be to the case when $G_*$ is the lower central series of $G$, we spell out the application of our general criterion to this case in Cor.~\ref{Andreadakis_for_Inn_LCS}. We then turn to examples, including the easy case of the free group. Our most prominent application will be to the pure braid group, which will be the goal of the next section.

\subsection{A general criterion}

Recall that for any element $g$ of a group $G$, the inner automorphism $c_g$ associated to $g$ is defined by $c_g(x) = {}^g\!x$ ($= gxg^{-1}$) for all $x \in G$. The map $c : g \mapsto c_g$ is a group morphism whose image is the normal subgroup $\Inn(G)$ of $\Aut(G)$. The kernel of this morphism is the set of elements $g \in G$ such that for all $x \in G$, $gxg^{-1} = x$, which is the center $\mathcal Z(G)$ of $G$. As a consequence, $\Inn(G) \cong G/\mathcal Z(G)$.

The same story can be told for Lie algebras: a Lie algebra $\mathfrak g$ acts on itself \emph{via} the adjoint action, the image of this action $\ad : \mathfrak g \rightarrow \Der(\mathfrak g)$ is the Lie algebra $\ad(\mathfrak g)$ of inner derivations, and the kernel of $\ad$ is the center $\mathfrak z(\mathfrak g)$ of $\mathfrak g$, so that $\ad(\mathfrak g) \cong \mathfrak g/\mathfrak z(\mathfrak g)$.

We now explain how these two stories are related to the Andreadakis problem. Let $G_*$ be a filtration on $G = G_1$, and let $\mathcal A_*(G_*)$ by the associated Andreadakis filtration. The filtration $G_*$ acts on itself \emph{via} the adjoint action, which is induced by the action of $G$ on itself by inner automorphisms. The latter is represented by the morphism $c : G \rightarrow \Aut(G)$ described above, and the fact that it induces an action of $G_*$ on itself is reflected in the fact that $c$ sends $G_*$ to $\mathcal A_*(G_*)$. Thus, if we consider the image of $G_*$ under the corestriction $\pi : G \twoheadrightarrow \Inn(G)$ of $c$, we get an inclusion of filtration $\pi(G_*) \subseteq \mathcal A_*(G_*)$. The following theorem gives a criterion for the inclusion $\pi(G_*) \subseteq \Inn(G) \cap \mathcal A_*(G_*)$ to be an equality:

\begin{theo}\label{Andreadakis_for_Inn}
Let $G_*$ be a filtration on $G = G_1$. Its image in $\Inn(G)$ coincides with $\Inn(G) \cap \mathcal A_*(G_*)$ if and only if the inclusion $\mathfrak z(\Lie(G_*)) \supseteq \Lie\left(G_* \cap \mathcal Z(G)\right)$ is an equality, that is, exactly when every central element of $\Lie(G_*)$ is the class of some central element of $G$.
\end{theo}

If $G_*$ is the lower central series of $G$, whose image in $\Inn(G)$ is the lower central series of $\Inn(G)$, then $\mathcal A_*(\Gamma_*G)$ is the usual Andreadakis filtration on $IA(G)$. 

\begin{cor}\label{Andreadakis_for_Inn_LCS}
For a group $G$, the following conditions are equivalent:
\begin{itemize}[itemsep=-3pt,topsep=3pt]
\item The Andreadakis equality holds for $\Inn(G)$.
\item Every central element of $\Lie(G)$ is the class of some central element of $G$.
\item The canonical projection $ \Lie \left(G/\mathcal Z(G)\right) \twoheadrightarrow \Lie(G)/\mathfrak z(\Lie(G)) = \ad(\Lie(G))$ is an isomorphism.
\end{itemize}
\end{cor}

This introduces a motivation for solving the Andreadakis problem for inner automorphisms of a group : it allows one to compute the Lie algebra of the quotient $G/\mathcal Z (G)$. We will apply this to the pure braid group $P_n$ later (\cref{par_Inn(Pn)}).

\begin{proof}[Proof of Cor.~\ref{Andreadakis_for_Inn_LCS}]
As a direct application of Theorem~\ref{Andreadakis_for_Inn}, the first assertion is equivalent to the inclusion $\mathfrak z(\Lie(G)) \supseteq \Lie\left(\Gamma_*(G) \cap \mathcal Z(G)\right)$ being an equality, which is clearly equivalent to the second assertion. In order to see that it is also equivalent to the third condition, let us consider the short exact sequence of filtrations :
\[\begin{tikzcd} 
\Gamma_*(G) \cap \mathcal Z(G) \ar[r, hook]  
&\Gamma_*(G) \ar[r, two heads] 
&\Gamma_*\left(G/\mathcal Z(G) \right).
\end{tikzcd}\]
By applying the Lie functor, we get an isomorphism:
\[\Lie \left(G/\mathcal Z(G) \right) \cong \Lie(G)/ \Lie(\Gamma_*(G) \cap \mathcal Z(G)).\]
As a consequence, the inclusion $\mathfrak z(\Lie(G)) \supseteq \Lie\left(\Gamma_*(G) \cap \mathcal Z(G)\right)$ induces a canonical projection $ \Lie \left(G/\mathcal Z(G)\right) \twoheadrightarrow \Lie(G)/\mathfrak z(\Lie(G))$. The latter is an isomorphism if an only if the former is an equality.
\end{proof}

\begin{proof}[Proof of Theorem \ref{Andreadakis_for_Inn}]
Let us denote by $\pi(G_*)$ the image of $G_*$ under $\pi : G \twoheadrightarrow \Inn(G)$, and the inclusion $\pi(G_*) \subseteq \Inn(G) \cap \mathcal A_*(G_*)$ by $i$. The latter induces a morphism $i_* : \Lie(\pi(G_*)) \rightarrow  \Lie\left(\Inn(G) \cap \mathcal A_*(G_*)\right)$, whose injectivity is equivalent to the two filtrations on $\Inn(G)$ being equal. By definition of $\pi$ and $i$, we have a commutative square:
\[\begin{tikzcd}
\pi(G_*) \ar[d, "i"] &
G_* \ar[d, "c"] \ar[l, swap, two heads, "\pi"]&\\
\Inn(G) \cap \mathcal A_*(G_*) \ar[r, hook] &
\mathcal A_*(G_*).
\end{tikzcd}\]
By taking the associated graded, we get the left square in:
\[\begin{tikzcd}
\Lie(\pi(G_*)) \ar[d, "i_\#"] &
\Lie(G_*) \ar[d, "c_\#"] \ar[rd, "\ad"] \ar[l, swap, two heads, "\pi_\#"]&\\
\Lie\left(\Inn(G) \cap \mathcal A_*(G_*)\right) \ar[r, hook] &
\Lie(\mathcal A_*(G_*)) \ar[r, hook, swap, "\tau"] &
\Der\left(\Lie(G_*)\right).
\end{tikzcd}\]
The fact that the triangle on the right commutes can be seen \emph{via} an abstract argument ($c$ represents the adjoint action of $G_*$, hence $c_\#$ represents the adjoint action of $\Lie(G_*)$), or can be obtained \emph{via} a direct calculation using the usual explicit description of the Johnson morphism : 
\begin{equation}\label{tau(c_g)}
\tau(\overline{c_g}) : \overline x \mapsto \overline{c_g(x)x^{-1}} =  \overline{[g,x]} = [\overline g, \overline x] = \ad_{\overline g}(\overline x).
\end{equation}
From this fact and the injectivity of the Johnson morphism, we deduce that the kernel of $c_\#$ is $\mathfrak z (\Lie(G_*))$. The Lie subring $\Lie\left(G_* \cap \mathcal Z(G)\right)$ of $\Lie(G_*)$, on the other hand, appears as the kernel of $\pi_\#$ by applying the Lie functor to the following short exact sequence of filtrations:
\[\begin{tikzcd}
G_* \cap \mathcal Z(G) \ar[r, hook] & 
G_* \ar[r, two heads, "\pi"] &
\pi(G_*).
\end{tikzcd}\]
Now, $\pi_\#$ induces a surjection from $\ker(c_\#) = \ker(i_\# \pi_\#)$ onto $\ker(i_\#)$ whose kernel is exactly $\ker(\pi_\#)$ (this can be seen as an application of the usual exact sequence $0 \rightarrow \ker(v) \rightarrow \ker(uv) \rightarrow \ker(u) \rightarrow \coker(v)$ to $u = i_\#$ and $v = \pi_\#$). Thus we have:
\begin{equation}\label{ker(i)}
\ker(i_\#) \cong \frac{\mathfrak z (\Lie(G_*))}{\Lie\left(G_* \cap \mathcal Z(G)\right)},
\end{equation}
whence the conclusion.
\end{proof}

\begin{rmq}
The isomorphism \eqref{ker(i)} gives more information than the statement of Theorem \ref{Andreadakis_for_Inn}, which is the case when this kernel is trivial. However, we will mainly use the latter case in the sequel.
\end{rmq}

\subsection{Examples}

We now apply Corollary \ref{Andreadakis_for_Inn_LCS} in order to give examples of groups whose group of inner automorphism satisfies (or not) the Andreadakis equality.

\subsubsection{Counter-examples}\label{par_counter_ex}

\paragraph{The symmetric group :} The first counter-examples we can give are groups $G$ with no center and a non-trivial abelian Lie ring (that is, a Lie ring reduced to $G^{ab} \neq 0$), such as the symmetric group $\Sigma_n$ ($n \geq 3$), whose abelianization is $\Sigma_n/A_n \cong \Z/2$. 

\paragraph{The braid group :} Another slightly more interesting counter-example in the braid group $B_n$, whose classical generators $\sigma_1, ..., \sigma_{n-1}$ are all conjugate. Its Lie algebra is reduced to its abelianization $B_n^{ab} \cong \Z$, generated by the common class $\overline \sigma$ of all $\sigma_i$. If $n \geq 3$, the center of $B_n$ is cyclic, generated by $\xi_n = (\sigma_1 \cdots \sigma_{n-1})^n$. Thus its image in $B_n^{ab} \cong \Z$ is not equal to all of $B_n^{ab}$, but to $n(n-1) \Z \subset \Z$.

\bigskip

Such examples are not very interesting to us, since filtrations on them do not contain a lot of information (they contain only information about the abelianization). We now describe a way of obtaining nilpotent (or residually nilpotent) counter-examples.

\paragraph{Constructing counter-examples as semi-direct products :} Let $G$ be a group, and $\alpha$ be an automorphism of $G$. Consider the semi-direct product $G \rtimes \Z$ encoding the action of $\Z$ on $G$ through powers of $\alpha$. This semi-direct product is an almost-direct one if and only if  $\alpha \in IA(G)$. Under this condition, $\Gamma_*(G \rtimes \Z) = \Gamma_*(G) \rtimes \Gamma_*(\Z)$, so that $G \rtimes \Z$ is nilpotent (resp.\ residually nilpotent) whenever $G$ is. 

Let us denote by $t$ the generator of $\Z$ acting \emph{via} $\alpha$ on $G$, and consider its class $\overline t$ in $\Lie(G \rtimes \Z)$. This element is central in $\Lie(G \rtimes \Z) = \Lie(G) \rtimes \Z$ if and only if its action on  $\Lie(G)$ is trivial.  This happens exactly when $\alpha \in \mathcal A_2(G)$. Indeed, for any class $\overline g \in \Lie_i(G)$, we have $[\overline t, \overline g] = \overline{[t, g]} = \overline{[\alpha,g]}$, and this is trivial in $\Lie_{i+1}(G)$ if and only if $[\alpha,g] \in \Gamma_{i+2}G$. Thus, $[\overline t, \overline g] = 0$ for all $\overline g$ if and only if for all $i \geq 1$, $[\alpha, \Gamma_iG]\subseteq \Gamma_{i+2}G$.

When is $\overline t$ the class of a central element in $G \rtimes \Z$ ? Lifts of $\overline t$ to $G \rtimes \Z$ are elements of the form $gt$ with $g \in \Gamma_2 G$. Using Proposition \ref{dec_of_centers} below, we see that such an element is central in $G \rtimes \Z$ if and only if $t$ acts on $G$ \emph{via} $c_{g^{-1}}$. Thus, $\overline t$  is the class of a central element of $G \rtimes \Z$ if and only if $\alpha$ is an inner automorphism $c_x$ ($x \in \Gamma_2G$), in which case $(x^{-1}, t)$ is the only central element whose class in $(G \rtimes \Z)$ is $\overline t$.

We conclude that this construction gives counter-examples whenever there exist automorphisms $\alpha$ in $\mathcal A_2$ which are not inner. Such automorphisms exist for the free group, or for the free nilpotent group of any class at least $3$, giving counter-examples of any nilpotency class greater than $3$.

\subsubsection{The free group}\label{par_Fn}

Let $G = F_n$ be the free group on $n$ generators, with $n \geq 2$. It is centerless, so that $\Inn(F_n) \cong F_n$.
We can apply the above machinery to the lower central series $\Gamma_*(F_n)$. The associated graded ring $\Lie(F_n)$ is free, whence centerless. As a consequence, we deduce directly from Corollary~\ref{Andreadakis_for_Inn_LCS} the following :

\begin{cor}\label{Andreadakis_for_Inn(Fn)}
The subgroup $\Inn(F_n)$ of $IA_n$ satisfies the Andreadakis equality.
\end{cor}

We do not need the full strength of the above theory of filtrations on inner automorphisms in order to get this result. In fact, it may be enlightening to write down a direct proof of this corollary.

\begin{proof}[Direct proof of Cor.~\ref{Andreadakis_for_Inn(Fn)}]
If $w \in F_n$, suppose that $c_w \in \mathcal A_j$, and let us show that $w \in \Gamma_j(F_n)$. The hypothesis means that:
\[\forall x \in F_n,\ \ c_w(x) \equiv x \pmod{\Gamma_{j+1}(F_n)}.\]
But $c_w(x)x^{-1} = [w,x]$, and $[w,x_i] \in \Gamma_{j+1}(F_n)$ implies that the class $\overline w$ of $w$ in $\Lie(F_n)$ either is of degree at least $j$ or commutes with the generator $\overline x_i$. The latter is possible only if $\overline w \in \mathbb Z \overline x_i$ (see the classical Lemma~\ref{centralizers_in_Ln} recalled below), which can be true only for one value of $i$. Thus $\overline w$ must be of degree at least $j$, which means that $w \in \Gamma_j(F_n)$.
\end{proof}

We can give yet another proof, which looks more like a simple version of the above theory in this particular case, and enhances the fact (needed later) that the Lie ring of $\Inn(F_n)$ identifies with the ring of inner derivations of $\mathfrak L_n$.

\begin{proof}[Yet another proof of Cor.~\ref{Andreadakis_for_Inn(Fn)}]
The Andreadakis equality is equivalent to the associated Johnson morphism $\tau: \Lie(F_n) \rightarrow \Der(\mathfrak L_n)$ being injective (Prop.~\ref{Johnson_inj}). Because of the direct calculation \eqref{tau(c_g)}, this morphism can be identified with $\ad: \mathfrak L_n \rightarrow \Der(\mathfrak L_n)$, whose kernel is the center of $\mathfrak L_n$, which is trivial. Thus the result.
\end{proof}

\section{Centers of semi-direct products}\label{section_centers}

We now turn to showing the Andreadakis equality for inner automorphisms of the pure braid group, a goal achived in Theorem \ref{Andreadakis_for_Inn(Pn)}. In order to do this, we recover the classical computation of the center of $P_n$ in a way that can be adapted easily to a computation of the center of its Lie ring $DK_n$.

\subsection{General theory}

The following easy result, which we have already used once in \ref{par_counter_ex}, will be our main tool in computing the center of $P_n$, and that of $DK_n$:

\begin{prop}\label{dec_of_centers}
Let $K$ be a group acting on another group $H$ by automorphisms. Then the center of $H \rtimes K$ consists of elements $hk$ (with $h \in H$ and $k \in K$) such that:
\begin{itemize}[itemsep=-3pt,topsep=3pt]
\item $k$ is central in $K$,
\item $h$ is a fixed point of the action of $K$ on $H$ (that is, $K \cdot h = \{h\}$),
\item $k$ acts on $H$ \emph{via} $c_h^{-1}$.
\end{itemize}
\smallskip
Similarly, let $\mathfrak k$ be a Lie algebra acting on another Lie algebra $\mathfrak h$ by derivations. Then the center of  $\mathfrak h \rtimes \mathfrak k$ consists of elements $x+y$ (with $x \in \mathfrak h$ and $y \in \mathfrak k$) satisfying:
\begin{itemize}[itemsep=-3pt,topsep=3pt]
\item $y$ is central in $\mathfrak k$,
\item $x$ is a fixed point of the action of $\mathfrak k$ on $\mathfrak h$ (that is, $\mathfrak k \cdot x = \{0\}$),
\item $y$ acts on $\mathfrak h$ \emph{via} $- \ad(x)$.
\end{itemize}
\end{prop}

\begin{proof}
Let $h \in H$ and $k \in K$. The condition $hk \in \mathcal Z (H \rtimes K)$ means exactly that for all $h' \in H$ and all $k' \in K$, we have $hkh'k' = h'k'hk$. For $h' = 1$, this gives $hkk' = k'hk = k'hk'^{-1}k'k$, from which we deduce $h = k'hk'^{-1}$ and $kk' = k'k$. This must be true for all $k' \in K$, whence the first two conditions. Suppose that these are satisfied. Then we write the condition for $hk$ to belong to $\mathcal Z (H \rtimes K)$ as $hkh'k^{-1}kk' = h'k'hk'^{-1}k'k$, which becomes $hkh'k^{-1} = h'h$, that is $k \cdot h' = h^{-1}h'h$ (which has to hold for all $h' \in H$). The latter is exactly the third condition in our statement.

The case of Lie algebras is quite similar. Let $x \in \mathfrak h$ and $y \in \mathfrak k$. The element $x+y$ belongs to $\mathfrak z (\mathfrak h \rtimes \mathfrak k)$ if and only if for all  $x' \in \mathfrak h$ and $y' \in \mathfrak k$, the bracket $[x+y,x'+y'] = [x,x'] +[y,y'] + y \cdot x' - y' \cdot x$ is trivial. For $x' = 0$, this gives $[y,y'] = 0$ and $y' \cdot x$, for all $y' \in \mathfrak k$, whence the first two conditions. Suppose that these are satisfied. Then the condition for $x+y$ to belong to $\mathfrak z (\mathfrak h \rtimes \mathfrak k)$ can be written as $[x,x'] + y \cdot x' = 0$, that is, $y \cdot x' = -[x,x']$ (which has to hold for all $x' \in \mathfrak h$). 
\end{proof}

\subsection{The center of the pure braid group}

Consider the pure braid group $P_n$ on $n$ generators. By applying Proposition \ref{dec_of_centers} to the semi-direct product decomposition $P_n \cong F_{n-1} \rtimes P_{n-1}$ recalled in \cref{reminders_braids}, we are able to recover the classical result of \cite{Chow}, which can also be found as \cite[Cor.~1.8.4]{Birman}:
\begin{prop}\label{Z(Pn)}
The center of $P_n$ (whence of $B_n$ if $n \geq 3$) is cyclic, generated by the element $\xi_n$ defined by $\xi_1 = 1$ and $\xi_{n+1} = \partial_n \cdot \xi_n$. As a braid automorphism, $\xi_n =  c_{\partial_n}^{-1}$.
\end{prop}

\begin{rmq}\label{formula_xi_n}
The relation $\xi_{n+1} = \partial_n \cdot \xi_n$ gives the usual formula for this central element \cite[cor. 1.8.4]{Birman}:
\[\xi_n = (A_{1n} A_{2n} \cdots A_{n-1,n})(A_{1, n-1} \cdots A_{n-2,n-1}) \cdots (A_{13}A_{23})  A_{12}.\]
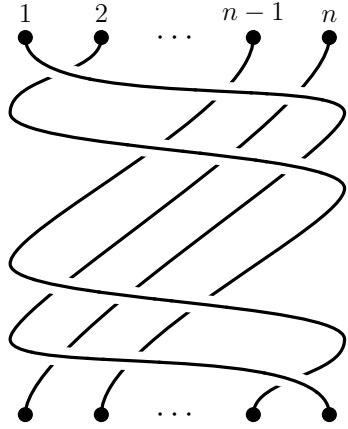
\[\begin{tikzpicture}
\node[draw=none] at (2,5) {$\cdots$};
\node[draw=none] at (2,0) {$\cdots$};
\begin{knot}[clip width = 6, flip crossing/.list={1,3,5,7,9,11}]
\strand[very thick]  
   (0, 0)  .. controls  +(0, 1) and +(0, -0.5) .. (4.2, 4)
   node[at start,circle, fill, inner sep=2pt]{}
   .. controls  +(0, 0.5) and +(0, -1) .. (0, 5)
   node[at end,circle, fill, inner sep=2pt]{}
   node[at end, above=2, scale=0.85]{$1$};
\strand[very thick]  
   (1, 0)  .. controls  +(0, 1) and +(0, -0.5) .. (4.2, 3)
   node[at start,circle, fill, inner sep=2pt]{} 
   .. controls  +(0, 0.5) and +(0, -0.5) .. (-.2, 4) 
   .. controls  +(0, 0.5) and +(0, -0.5) .. (1, 5)
   node[at end,circle, fill, inner sep=2pt]{}
   node[at end, above=2, scale=0.85]{$2$};
\strand[very thick]  
   (3, 0) .. controls  +(0, 0.5) and +(0, -0.5) .. (4.2, 1)
   node[at start,circle, fill, inner sep=2pt]{}
   .. controls  +(0, 0.5) and +(0, -0.5) .. (-.2, 2) 
   .. controls  +(0, 0.5) and +(0, -1) .. (3, 5)
   node[at end,circle, fill, inner sep=2pt]{}
   node[at end, above=2, scale=0.85]{$n-1$};
\strand[very thick]  
   (4, 0)  .. controls  +(0, 1) and +(0, -0.5) .. (-.2, 1)
   node[at start,circle, fill, inner sep=2pt]{}
   .. controls  +(0, 0.5) and +(0, -1) .. (4, 5)
   node[at end,circle, fill, inner sep=2pt]{}
   node[at end, above=2, scale=0.85]{$n$};
\end{knot}
\end{tikzpicture}\]
\end{rmq}

\begin{proof}[Proof of Prop.~\ref{Z(Pn)}]
The result is clear for $n \leq 2$. We thus suppose $n \geq 3$.

Lemma \ref{Pn_cap_Fn} implies that $\mathcal Z(P_n)$ is not trivial: it contains the braid automorphism~$c_{\partial_n}$. We use the decomposition $P_n \cong F_{n-1} \rtimes P_{n-1}$ to show that this element generates $\mathcal Z(P_n)$. Let $w \beta \in \mathcal Z(F_{n-1} \rtimes P_{n-1})$. Proposition \ref{dec_of_centers} implies that $\beta$ should act on $F_{n-1}$ \emph{via} $c_w^{-1}$. From Lemma \ref{Pn_cap_Fn} , we deduce that $w \in \langle \partial_{n-1} \rangle$. As a consequence:
\begin{equation}\label{incl_ZPn}
\mathcal Z(P_n) \subseteq \left\langle \partial_{n-1} c_{\partial_{n-1}}^{-1} \right\rangle.
\end{equation}

Then, the central element $c_{\partial_n}$ must be equal to $\partial_{n-1}^k c_{\partial_{n-1}}^{-k}$ for some $k$. But since the projection from $P_n$ onto $P_{n-1}$ is induced by $x_n \mapsto 1$, we see that it sends $c_{\partial_n}$ to $c_{\partial_{n-1}}$, whence $k = 1$, and \eqref{incl_ZPn} is an equality. Moreover, we have obtained the induction relation $c_{\partial_n} = \partial_{n-1}^{-1} c_{\partial_{n-1}}$, which implies the relation we wanted for $\xi_n$, if we define $\xi_n$ to be $c_{\partial_n}^{-1}$.

Finally, we remark that for $n \geq 3$, the center of $B_n$ has to be a subgroup of $P_n$ (whence of $\mathcal Z(P_n)$), since its image in $\Sigma_n$ must be in the center of $\Sigma_n$, which is trivial. The other inclusion can be obtained by showing directly that $\xi_n$ commutes with the classical generators of $B_n$ (which is obvious from the geometric picture).
\end{proof}

An element of the free group is called \emph{primitive} when it is part of a basis of the free group. Recall the following easy result:

\begin{lem}\label{centralizers_in_Fn}
In the free group $F_n$, the centralizer of any primitive element $w$ is the cyclic group generated by $w$.
\end{lem}

\begin{proof}
Suppose that $(x_1 = w, x_2, ..., x_n)$ is a basis of $F_n$. If $g \in F_n$, the relation $x_1 g x_1^{-1}g^{-1} = 1$ cannot hold in the free group if any letter different from $x_1$ appears in the reduced expression of $g$ in the letters $x_i$, whence our claim.
\end{proof}

\begin{lem}\label{Pn_cap_Fn}
In $\Aut(F_n)$, the intersection between $P_n$ and $\Inn(F_n)$ is cyclic, generated by $c_{\partial_n}$, which is a central element of $P_n$.
\end{lem}

\begin{proof}
Let $\beta \in P_n$ be a braid automorphism which is also an inner automorphism $c_w$ for some $w \in F_n$. Then $c_w$ must fix $\partial_n$, that is, $w$ must commute with $\partial_n$. However, $\partial_n$ is a primitive element of $F_n$. Indeed, $(\partial_n, x_2, ..., x_n)$ is a basis of $F_n$, the change of basis being given by $\partial_n = x_1 \cdots x_n$ and $x_1 = \partial_n x_n^{-1} \cdots x_2^{-1}$. As a consequence, $w \in \langle \partial_n \rangle$. 
Thus $P_n \cap \Inn(F_n) \subseteq \langle c_{\partial_n} \rangle$. This inclusion is in fact an equality, since $c_{\partial_n}$ is a braid automorphism. Moreover, for any braid automorphism $\beta \in P_n$, we have $\beta c_{\partial_n} \beta^{-1} = c_{\beta(\partial_n)} = c_{\partial_n}$, hence $c_{\partial_n}$ is central in $P_n$.
\end{proof}

\subsection{The center of the Drinfeld-Kohno Lie ring}\label{par_Z(DKn)}

The above computation of $\mathcal Z(P_n)$ can readily be adapted to compute the center of the Lie ring of $P_n$, which is the Drinfeld-Kohno Lie ring $DK_n$. We use the decomposition $DK_n \cong \mathfrak L_n \rtimes DK_{n-1}$, induced by the decomposition $P_n \cong F_{n-1} \rtimes P_{n-1}$ (see \cref{reminders_DK}). The analogue of Lemma \ref{centralizers_in_Fn} holds in the free Lie ring:

\begin{lem}\label{centralizers_in_Ln}
In the free Lie ring $\mathfrak L_n$, the centralizer of an element $x \in \mathbb Z^n$ (of degree $1$) is $\mathbb Z \cdot \frac xd$, where $d$ is the $\gcd$ of the coefficients of $x$.
\end{lem}

\begin{proof}
Consider a basis $(x_1 = x/d, x_2, ..., x_n)$ of $\mathbb Z^n$. 
In the tensor algebra $T(\mathbb Z^n)$, which is the enveloping ring of $\mathfrak L_n$, the centralizer of $d x_1$ consists of all polynomials in $x_1$ only. Among these, the only ones belonging to $\mathfrak L_n$ (the only \emph{primitive} ones, in the sense of Hopf algebras) are the linear ones.
\end{proof}

We also have an analogue of Lemma \ref{Pn_cap_Fn} in this context (see Def.~\ref{defi_braid_der} for the definition of braid-like derivation): 
\begin{lem}\label{DKn_cap_Ln}
In $\Der(\mathfrak L_n)$, the intersection between $\Der^\partial_t(\mathfrak L_n)$ and $\Inn(\mathfrak L_n)$ is cyclic, generated by $\ad_{\overline \partial_n}$, which is a central element of $\Der^\partial_t(\mathfrak L_n)$.
\end{lem}

\begin{proof}
Let $X \in \mathfrak L_n$ such that the inner derivation $\ad_X$ is braid-like. Then $\ad_X(\overline \partial_n) = 0$, that is, $X$ must be in the centralizer of $\overline\partial_n$. We deduce from Lemma \ref{centralizers_in_Ln} that $X \in \mathbb Z \cdot \overline\partial_n$. 
Thus $\Der^\partial_t(\mathfrak L_n) \cap \ad(\mathfrak L_n) \subseteq \mathbb Z \cdot \overline\partial_n$. This inclusion is in fact an equality, since $\ad_{\overline \partial_n}$ is a braid-like derivation. Moreover, for any braid-like derivation $d$, we have:
\[ \left[d,\ad_{\overline \partial_n}\right] = d \left(\left[\overline \partial_n, -\right]\right) - \left[\overline \partial_n, d(-)\right] = \left[ d\left(\overline \partial_n\right), -\right] = 0,\] hence $\ad_{\overline \partial_n}$ is central in $\Der^\partial_t(\mathfrak L_n)$.
\end{proof}

Finally, we can use Lemmas \ref{centralizers_in_Ln} and \ref{DKn_cap_Ln} to compute the center of $DK_n$:
\begin{prop}\label{Z(DKn)}
The center of $DK_n$ is cyclic, generated by the element $\overline \xi_n$ defined by $\overline \xi_1 = 0$ and $\overline \xi_{n+1} = \overline \partial_n  + \overline \xi_n$. As a braid derivation, $\overline \xi_n =  -\ad(\overline \partial_n)$.
\end{prop}

\begin{rmq}
Out of the relation $\overline \xi_{n+1} = \overline \partial_n  + \overline \xi_n$, we get a formula for this central element:
\[\overline  \xi_n = \sum\limits_{i < j} t_{ij} \in DK_n.\]
If we compare this formula with that of remark \ref{formula_xi_n}, we can see directly that $\overline  \xi_n$ is the class of the central element $\xi_n$ of $P_n$.
\end{rmq}

\begin{proof}[Proof of Prop. \ref{Z(Pn)}]
The result is trivial for $n \leq 2$. We thus suppose $n \geq 3$.

We first remark that the braid-like derivation $\ad_{\overline \partial_n}$ of Lemma \ref{DKn_cap_Ln} is in fact a braid derivation. Indeed, $DK_n$ contains all braid-like derivations of degree $1$ (see Prop.~\ref{DK_eng_deg_1}). Alternatively, we can use formula \eqref{tau(c_g)} to see that $\ad_{\overline \partial_n}$ identifies with the class of the braid automorphism $c_{\partial_n}$ \emph{via} the Johnson morphism $\tau: DK_n \hookrightarrow \Der^\partial_t(\mathfrak L_n)$.

Since $DK_n$ identifies with a Lie subring of $\Der^\partial_t(\mathfrak L_n)$, Lemma \ref{DKn_cap_Ln} implies that $\mathfrak z(DK_n)$ is not trivial: it contains the braid derivation~$\ad_{\overline \partial_n}$. We use the decomposition $DK_n \cong \mathfrak L_{n-1} \rtimes DK_{n-1}$ to show that this element generates $\mathfrak z(DK_n)$. Let $X + d \in \mathfrak L_{n-1} \rtimes DK_{n-1}$. Proposition \ref{dec_of_centers} implies that $d$ should act on $\mathfrak L_{n-1}$ \emph{via} $-\ad_X$. From Lemma \ref{DKn_cap_Ln} , we deduce that $X \in \Z \cdot \overline \partial_{n-1}$. As a consequence:
\begin{equation}\label{incl_ZDKn}
\mathfrak z(DK_n) \subseteq \Z \cdot (\overline \partial_{n-1} -\ad_{\overline \partial_{n-1}}) \subset \mathfrak L_{n-1} \rtimes DK_{n-1}.
\end{equation}

The central element $\ad_{\overline \partial_n}$ must then be equal to $k \cdot (\overline \partial_{n-1} -\ad_{\overline \partial_{n-1}})$ for some integer $k$. But since the projection from $DK_n$ onto $DK_{n-1}$ is induced by $X_n \mapsto 0$, we see that it sends $\ad_{\overline \partial_n}$ to $\ad_{\overline \partial_{n-1}}$, whence $k = -1$, and  \eqref{incl_ZDKn} is an equality. Moreover, we have obtained the induction relation $\ad_{\overline \partial_n} = - \overline \partial_{n-1} + \ad_{\overline \partial_{n-1}}$, which implies the relation we wanted for $\overline \xi_n$, if we define $\overline  \xi_n$ to be $- \ad_{\overline \partial_n}$.
\end{proof}

\subsection{Inner automorphisms of the pure braid group}\label{par_Inn(Pn)}

Let us denote, as usual, the quotient $P_n/ \mathcal Z(P_n)$ by $P_n^*$. In a similar fashion, we denote by $DK_n^*$ the quotient of the Drinfeld-Kohno Lie ring $DK_n$ by its center.

\begin{theo}\label{Andreadakis_for_Inn(Pn)}
The subgroup $\Inn(P_n)$ of $IA(P_n)$  ($\subset \Aut(P_n)$) satisfies the Andrea\-dakis equality. Equivalently:
\[\Lie(P_n^*) \cong DK_n^*.\]
\end{theo}

\begin{proof}
The only central elements of $DK_n$ are multiples of $\overline \xi_n$, which is the class of the central element $\xi_n \in P_n$  (Prop.~\ref{Z(DKn)}). Thus, our claim is obtained as a direct application of Cor.~\ref{Andreadakis_for_Inn_LCS} to $G = P_n$. 
\end{proof}

\begin{cor}\label{Andreadakis_for_Inn(Pn*)}
The Lie algebra $DK_n^*$, and hence the group $P_n^*$, are centerless. As a consequence, the subgroup $\Inn(P_n^*)$ of $IA(P_n^*)$ also satisfies the Andreadakis equality. 
\end{cor}

\begin{proof}
We use the fact that the only non-trivial central elements of $DK_n$ are in degree one. Let $\overline x \in \mathfrak z(DK_n^*)$ be the class of $x \in DK_n$. Let $y$ be any element of $DK_n$. Then the bracket $[x,y]$ must be in $\mathfrak z(DK_n)$ (since its class in $DK_n^*$ is trivial). But since its degree is at least $2$, it must be trivial. Hence $x \in \mathfrak z(DK_n)$, whence $\overline x = 0$.

In order to deduce that $P_n^*$ is centerless, we can use the fact that it is residually nilpotent, since $P_n$ is (Lemma \ref{G/Z_res_nilp}), and thus any non-trivial element in its center would give a non-trivial class in its Lie algebra, which must be central too.

The last statement is then a direct application of Cor.~\ref{Andreadakis_for_Inn_LCS} to $G = P_n^*$.
\end{proof}

\begin{rmq}\label{rmq_on_splitting}
Part of the results of Th.~\ref{Andreadakis_for_Inn(Pn)} and its corollary can be deduced from the (classical) calculation of the center of $P_n$ and from the classical (non-canonical) splitting $P_n \cong P_n^* \times \mathcal Z(P_n)$ recalled in the appendix (Cor.~\ref{splitting_of_center_for_Pn}). In fact, if we apply the Lie functor to this direct product, we get:
\[DK_n = \Lie(P_n) \cong \Lie(P_n^*) \times  \Lie(\mathcal Z(P_n)) \cong \Lie(P_n^*) \times \Z \cdot \overline \xi_n,\]
where the right factor is the abelian Lie subring generated by the class $\overline \xi_n = \sum t_{ij}$ of the generator $\xi_n$ of $\mathcal Z(P_n)$. Thus we get the computation $\Lie(P_n^*) \cong DK_n/\overline \xi_n$. Moreover, we see easily that $P_n^*$ is centerless (Prop.~\ref{splitting_of_center}). However, neither does this imply that $\overline \xi_n$ generates the center of $DK_n$, nor do we get the statements about Andreadakis equalities without computing this center.
\end{rmq}

\section{Products of subgroups of \texorpdfstring{$IA_n$}{IAn}}\label{section_products}

Here we turn to the proof of our key result (Th.~\ref{key_theo}), which generalises \cite[Prop.~4.1]{Darne2}. In order to do this, we study filtrations on products $HK$, where $H$ and $K$ are subgroups of a given group $G$, such that $K$ normalizes $H$. Namely, we investigate the behaviour of the lower central series of $HK$, and of a filtration $G_* \cap (HK)$ induced by a filtration $G_*$ of $G$, with respect to the product decomposition of $HK$.

\subsection{Lower central series of products of subgroups}\label{par_LCS_of_HK}

Let $G$ be a group, and let $H$ and $K$ be subgroups of $G$, such that $K$ normalizes $H$. Then $K$ acts on $H$ by conjugation in $G$, we can form the corresponding semi-direct product $H \rtimes K$ and we get a well-defined morphism
$H \rtimes K \rightarrow G$
given by $(h,k) \mapsto hk$.  The image of this morphism is the subgroup $HK$ of $G$. Its kernel, given by the elements $(h,k) \in H \rtimes K$ such that $hk = 1$, is isomorphic to $H \cap K$, via $k \mapsto (k^{-1},k)$. Thus we get a short exact sequence of groups:
\[H \cap K \hookrightarrow H \rtimes K \twoheadrightarrow HK.\]
The surjection on the right induces a surjection $\Gamma_i(H \rtimes K) \twoheadrightarrow \Gamma_i(HK)$, for all $i \geq 1$. Since $\Gamma_i(H \rtimes K) = \Gamma_i^K(H) \rtimes \Gamma_i(K)$ (see Prop.~\ref{lcs_of_sdp}), this implies:
\begin{prop}
Let $K$ and $H$ be subgroups of a group $G$, such that $K$ normalizes $H$. Then:
\[\forall i \geq 1,\ \Gamma_i(HK) = \Gamma_i^K(H)\Gamma_i(K).\]
\end{prop} 
Moreover, the kernel of $\Gamma_i(H \rtimes K) \twoheadrightarrow \Gamma_i(HK)$ consists of elements $k \in H \cap K$ such that $(k^{-1},k) \in \Gamma_i(H \rtimes K)$: it is $\Gamma_i^K(H) \cap \Gamma_i(K) \subseteq H \cap K$. Thus we get a short exact sequence of filtrations:
\[\Gamma_*^K(H) \cap \Gamma_*(K) \hookrightarrow \Gamma_*^K(H) \rtimes \Gamma_*(K)\twoheadrightarrow \Gamma_*(HK).\]

\subsection{Decomposition of an induced filtration}\label{par_restr_to_HK}

Let $G_*$ be a filtration, and let $H$ and $K$ be subgroups of $G = G_1$ such that $K$ normalizes $H$. Then we can consider two filtrations on $HK$: the induced filtration $G_* \cap (HK)$ and the product of induced filtrations $(G_* \cap H)(G_* \cap K)$. The former obviously contains the latter. We now describe a criterion for this inclusion to be an equality.

\begin{prop}\label{Dec_induced_filtration}
In the above setting, the following assertions are equivalent:
\begin{enumerate}[label={(\roman*)}]
\item \label{item1} $G_* \cap (HK) = (G_* \cap H)(G_* \cap K),$
\item \label{item2} Inside $\Lie(G_*)$,\ \ $\Lie(G_* \cap H) \cap \Lie(G_* \cap K) = \Lie\left(G_* \cap (H \cap K)\right).$
\end{enumerate}
\end{prop}

\begin{rmq}
The case considered in \cite[Prop.~4.1]{Darne2} was exactly the case when $H \cap K = 1$. In this context, $H$ and $K$ were called \emph{$G_*$-disjoint} when they satisfied the equivalent conditions of the proposition.
\end{rmq}

\begin{proof}[Proof of Proposition \ref{Dec_induced_filtration}]

Suppose that \ref{item2} does not hold. Then there exists an element $x \in \Lie(G_* \cap H) \cap \Lie(G_* \cap K)\ -\ \Lie\left(G_* \cap (H \cap K)\right)$ for some $i \geq 1$. Then $x = \overline h = \overline k$ for some $h \in G_i \cap H$ and some $k \in G_i \cap K$. Since $\overline h = \overline k$ in $\Lie_i(G_*)$, the element $g = h^{-1}k$ is in $G_{i+1}$. It is also obviously in $HK$. However, we claim that $g \notin (G_{i+1} \cap H)(G_{i+1} \cap K)$. Indeed, if we could write $g$ as a product $h'k'$ with $h' \in G_{i+1} \cap H$ and  $k' \in G_{i+1} \cap K$, then we would have $h^{-1}k = h'k'$, whence $k k'^{-1} = hh' \in H \cap K$. And by construction, $\overline{hh'} = \overline h = x$ in $G_i/G_{i+1}$. But this would imply that $x \in \Lie\left(G_* \cap (H \cap K)\right)$, a contradiction. Thus $g$ must be a counter-example to \ref{item1}.

Conversely, suppose \ref{item1} false. Then there exists $g \in G_j \cap (HK)$, such that for all $(h,k) \in H \rtimes K$ satisfying $g = hk$, neither $h$ nor $k$ belongs to $G_j$ (if $h$ or $k$ belongs to $G_j$, so does the other one, since their product $g$ does). Then $h \equiv k^{-1} \not\equiv 1 \pmod{G_j}$. For all such $(h,k)$, there exists $i < j$ such that $h, k \in G_i - G_{i+1}$. Let us take $(h,k)$ such that this index $i$ is maximal. We show that the element $\overline h = - \overline k \in \Lie_i(G_*)$ gives a counter-example to the equality \ref{item2}. Indeed, it is clear that $\overline h = - \overline k$ belongs to $\Lie_i(G_* \cap H) \cap \Lie_i(G_* \cap K)$. Suppose now that this element would belong to $\Lie_i(G_* \cap (H \cap K))$. Then there would exist $x \in G_i \cap H \cap K$ such that $\overline h = - \overline k = \overline x$. This means that $\overline h - \overline x = - (\overline x + \overline k ) = 0$ or, equivalently: $\overline{hx^{-1}} = - \overline{xk} = 0$. But then $g = hk = (hx^{-1})(xk)$, with $hx^{-1}$ and  $xk$ in $G_{i+1}$, contradicting the maximality of $i$. Thus $\overline h$ cannot belong to $\Lie_i(G_* \cap (H \cap K))$, whence our claim.
\end{proof}

\subsection{Application to the Andreadakis problem}

We are now able to state our key theorem, which is an improvement on \cite[Th.~4.2]{Darne2}. We will apply it to the case when $G = IA_n$ and $G_* = \mathcal A_*$ is the Andreadakis filtration, but we still give a general statement.

\begin{theo}\label{key_theo}
Let $G_*$ be a filtration, and let $H$ and $K$ be subgroups of $G = G_1$ such that $[K,H] \subseteq [H,H]$. Suppose that in the Lie ring $\Lie(G_*)$, the intersection of the Lie subrings $\Lie(G_* \cap H)$ and $\Lie(G_* \cap K)$ is $\Lie\left(G_* \cap (H \cap K)\right)$. Then: 
\[
\begin{cases}
G_* \cap K = \Gamma_*(K) \\ G_* \cap H = \Gamma_*(H)
\end{cases} 
\Rightarrow\ \ G_* \cap (HK) = \Gamma_*(HK).\]
\end{theo}

Recall that in the case when $G = IA_n$ and $G_* = \mathcal A_*$ is the Andreadakis filtration, the Lie ring $\Lie(\mathcal A_*)$ embeds into the Lie ring $\Der(\mathfrak L_n)$ of derivations of the free Lie ring, \emph{via} the Johnson morphism $\tau$. As a consequence, the hypothesis about Lie subrings of $\Lie(\mathcal A_*)$ can be checked there.

\begin{proof}[Proof of Theorem \ref{key_theo}]
Remark that $[K,H] \subseteq [H,H]$ implies in particular that $K$ normalizes $H$. Thus, we can apply the results of \cref{par_LCS_of_HK}. Since $[K,H] \subseteq [H,H]$ is exactly the condition needed for $\Gamma_*^K(H)$ to be equal to $\Gamma_*(H)$ (Prop.~\ref{lcs_of_adp}), we get:
\[\Gamma_*(HK) = \Gamma_*(H)\Gamma_*(K).\]
Then we use Proposition \ref{Dec_induced_filtration} to get:
\[G_* \cap (HK) = (G_* \cap H)(G_* \cap K),\]
whence the result.
\end{proof}

\section{Applications}

The subgroup $H = \Inn(F_n) \cong F_n$ of $IA_n$ is an ideal candidate for applying our key result (Th.~\ref{key_theo}). We begin by spelling out the consequences of Th.~\ref{key_theo} in this particular case (Cor.~\ref{cor_to_key_theo}). It turns out that for each of the subgroups $K$ which were shown to satisfy the Andreadakis equality in \cite{Darne2}, we can apply Cor.~\ref{cor_to_key_theo} to show that $HK$ does too. Moreover, two of the subgroups so obtained had already been considered in the literature. These are the group of \emph{partial inner automorphisms} of \cite{Bardakov-Neshchadim}, and the subgroup $P_{n+1}^*$ of $IA_n$ presented in the introduction, which was first studied in \cite{Magnus}.

\subsection{Adding inner automorphisms}

We are going to apply Theorem~\ref{key_theo} with $H = \Inn(F_n)$ and $G = IA_n$ endowed with the Andreadakis filtration $G_* = \mathcal A_*$, for three different subgroups $K$ of $IA_n$. We record in Corollary~\ref{cor_to_key_theo} below the consequences of Theorem~\ref{key_theo} in this context.

\begin{cor}[to Th.~\ref{key_theo}]\label{cor_to_key_theo}
Let $K$ be a subgroup of $IA_n$, and $\tau: \Lie(K) \rightarrow \Der(\mathfrak L_n)$ be the corresponding Johnson morphism. Suppose that $K$ satisfies the Andreadakis equality, which means that $\tau$ is injective. Suppose, moreover, that every element of $\tau(\Lie(K)) \cap \ad(\mathfrak L_n)$ comes from an element of $K \cap \Inn(F_n)$ (precisely, it equals $\tau(\overline x)$, for some $x \in K \cap \Inn(F_n)$). Then $\Inn(F_n)K$ satisfies the Andreadakis equality.
\end{cor}

\begin{proof}
We apply Theorem~\ref{key_theo} to the subgroups $H = \Inn(F_n)$ and $K$ of $G = IA_n$ endowed with the Andreadakis filtration $G_* = \mathcal A_*$. 

\begin{itemize}[itemsep=-3pt,topsep=3pt]
\item The hypothesis $[K,H] \subseteq [H,H]$ comes from the fact that $IA_n$ (whence $K$) normalizes $\Inn(F_n)$ and acts trivially on $F_n^{ab}$.
\item Cor.~\ref{Andreadakis_for_Inn(Fn)} says that $H = \Inn(F_n)$ satisfies the Andreadakis equality.
\item By hypothesis, $K$ does too.
\end{itemize}
Thus, we are left with proving that the following inclusion (which is true in general) is in fact an equality in $\Lie(\mathcal A_*)$:
\[\Lie(\mathcal A_* \cap \Inn(F_n)) \cap \Lie(\mathcal A_* \cap K) \supseteq \Lie\left(\mathcal A_* \cap (\Inn(F_n) \cap K)\right).\]
If we embed $\Lie(\mathcal A_*)$ into $\Der(\mathfrak L_n)$ \emph{via} the Johnson morphism, the Lie ring $\Lie(\mathcal A_* \cap \Inn(F_n))$ identifies with the Lie subring of inner derivations $\ad(\mathfrak L_n)$ (see \cref{par_Fn}). Moreover, since $K$ satisfies the Andreadakis equality, the Lie ring $\Lie(\mathcal A_* \cap K)$ identifies with $\tau(\Lie(K))$. As for $\Lie\left(\mathcal A_* \cap (\Inn(F_n) \cap K)\right)$, it is exactly the set of $\tau(\overline x)$, for $x \in K \cap \Inn(F_n)$, whence the result.
\end{proof}

\subsection{Triangular automorphisms}

We recall the definition of the subgroup $IA_n^+$ of triangular automorphisms \cite[Def.~5.1]{Darne2}:
\begin{defi}\label{def_triangulaires}
Fix $(x_1,..., x_n)$ an ordered basis of $F_n$. The subgroup $IA_n^+$ of $IA_n$ consists of \emph{triangular automorphisms}, \emph{i.e.\ } automorphisms $\varphi$ acting as:
\[\varphi: x_i \longmapsto (x_i^{w_i})\gamma_i,\]
where $w_i \in \langle x_j \rangle_{j < i} \cong F_{i-1}$ et $\gamma_i \in \Gamma_2(F_{i-1})$.
\end{defi}

Recall that $IA_n^+$ satisfies the Andreadakis equality. This statement first appeared as \cite[Th.~1]{Satoh_triangulaire}, and a shorter proof was given in \cite{Darne2}.

\begin{theo}\label{Andreadakis_for_FnIAn+}
The subgroup $\Inn(F_n) IA_n^+$ of $IA_n$ satisfies the Andreadakis equality.
\end{theo}

\begin{proof}
Remark that a triangular automorphism $\varphi$ has to fix $x_1$. As a consequence, the derivation $\tau(\overline \varphi)$ sends $x_1$ to $\overline{\varphi(x_1)x_1^{-1}} = 0$. The only inner derivations vanishing on $x_1$ are the multiples of $\ad(x_1)$ (Lemma~\ref{centralizers_in_Ln}). However, $\ad(x_1)$ is the image by $\tau$ of $\overline{c_{x_1}}$, and $c_{x_1}$ is a triangular automorphism. Thus, the hypotheses of  Corollary~\ref{cor_to_key_theo} are satisfied, whence the desired conclusion.
\end{proof}

\subsection{The triangular McCool group}

The triangular McCool $P \Sigma_n^+$ group was first considered in \cite{CPVW}, where its Lie algebra was computed. It was shown to satisfy the Andreadakis equality in \cite{Darne2}.

The subgroup $\Inn(F_n) P \Sigma_n^+$ of $IA_n$ (or, more precisely, $\Inn(F_n) P \Sigma_n^-$, which is obtained from the definition of $\Inn(F_n) P \Sigma_n^+$ with the opposite order on the generators) is exactly the \emph{partial inner automorphism group} $I_n$ defined and studied in \cite{Bardakov-Neshchadim}. Indeed, $I_n$ is defined as the subgroup generated by the automorphism $c_{ki}$ for $k \geq i$, defined by:
\[c_{ki}: x_j \longmapsto 
\begin{cases}
x_j^{x_i} &\text{if }\ j \leq k, \\
x_j &\text{else.}
\end{cases}\]
Equivalently, it is generated by the $c_{ni} = c_{x_i^{-1}}$ (which generate $\Inn(F_n)$), together with the $c_{ki}c_{k-1,i}^{-1} = \chi_{ki}$ for $k > i$, which generate $P \Sigma_n^-$.

\begin{theo}\label{Andreadakis_for_FnMcCool}
The subgroup $I_n = \Inn(F_n)P \Sigma_n^+$ of $IA_n$ satisfies the Andreadakis equality.
\end{theo}

\begin{proof}
The proof is exactly the same as the proof of Theorem \ref{Andreadakis_for_FnIAn+}, using the fact that $P \Sigma_n^+$ satisfies the Andreadakis equality \cite[Cor.~5.5]{Darne2}, and remarking that $c_{x_1}$ is not only in $IA_n^+$, but in $P \Sigma_n^+$.
\end{proof}

\subsection{The pure mapping class group of the punctured sphere}\label{par_Andreadakis_for_Pn*}

The Artin action of $P_n$ on $F_n$ is by group automorphisms fixing $\partial_n = x_1 \cdots x_n$. As a consequence, it induces an action of $P_n$ on the quotient $F_n/\partial_n \cong F_{n-1}$ (which is free on $x_1, ..., x_{n-1}$ since $(x_1, ..., x_{n-1}, \partial_n)$ is a basis of $F_n$).
This action is not faithful anymore, since the generator $c_{\partial_n}$ of $\mathcal Z (P_n)$ (see Prop.~\ref{Z(Pn)}) acts trivially modulo $\partial_n$. However, by a classical theorem of \cite{Magnus}, of which we give a simple proof in our appendix (\cref{par_action_Pn*}), the kernel is not bigger that $\mathcal Z (P_n)$ : the above induces a faithful action of $P_n/\mathcal Z(P_n) =: P_n^*$ on $F_{n-1}$.

The latter is exactly the action on the free group of the (pure, based) mapping class group of the punctured sphere  presented in the introduction, described in a purely algebraic fashion. Indeed, $P_n$ identifies with the boundary-fixing mapping class group of the $n$-punctured disc, and the Artin action corresponds to the canonical action on the fundamental group of this space. Then, collapsing the boundary to a point (which we choose as basepoint) to get an action on the $n$-punctured sphere corresponds exactly to taking the quotient by $\partial_n = x_1 \cdots x_n$, at the level of fundamental groups.

From the description of Artin's action, we see that automorphisms of $F_{n-1}$ obtained from the above action must send each generator $x_i$ to one of its conjugates and, since $x_n^{-1} \equiv x_1 \cdots x_{n-1} \pmod{\partial_n}$, they must also send the boundary element $\partial_{n-1} = x_1 \cdots x_{n-1}$ to one of its conjugates. This means exactly that this action is \emph{via} elements of $\Inn(F_{n-1})P_{n-1}$:

\begin{lem}\label{description_FnPn}
The subgroup $\Inn(F_n)P_n$ of $IA_n$ is the subgroup of all automorphisms of $F_n$ sending each generator $x_i$ to one of its conjugates, and sending the boundary element $\partial_n = x_1 \cdots x_n$ to one of its conjugates. 
\end{lem}

\begin{proof}
The conditions of the lemma obviously describe a subgroup $G$ of $\Aut(F_n)$ (it is an intersection of stabilizers for the action of $\Aut(F_n)$ on the set of conjugacy classes of $F_n$), and it contains $P_n$ and $\Inn(F_n)$. Now let $\sigma \in G$. Then $\sigma(\partial_n) = \partial_n^w$ for some $w$, hence $c_w \circ \sigma$ fixes $\partial_n$ (where $c_w: x \mapsto {}^w\!x$ is the inner automorphism associated to $w$). As a consequence, $c_w \sigma \in P_n$, whence $\sigma \in \Inn(F_n)P_n$, and our claim.
\end{proof}

In fact, the proof of Magnus' theorem (Th~\ref{Pn*_acts_faithfully}) that we give consists in showing that the above action of $P_n$ on $F_{n-1}$ induces an isomorphism between $P_n^*$ and $\Inn(F_{n-1})P_{n-1}$. 

We now give the main application of our key result (Th.~\ref{key_theo}):

\begin{theo}\label{Andreadakis_for_FnPn}
The subgroup $P_{n+1}^* \cong \Inn(F_n)P_n$ of $IA_n$ satisfies the Andreadakis equality.
\end{theo}

\begin{proof}
We apply Corollary~\ref{cor_to_key_theo} to the subgroup $K = P_n$ of $IA_n$. The Andreadakis equality holds for $P_n$: this was proved in \cite[Th.~6.2]{Darne2}, recalled as Th.~\ref{Andreadakis_for_Pn} above. We need to check the other hypothesis. Recall that $\Lie(P_n)$, which is the Drinfeld-Kohno Lie ring $DK_n$, is identified to a Lie subring of $\Der^\partial_t(\mathfrak L_n)$ as in \cref{par_Z(DKn)}. Moreover, the intersection of $\Der^\partial_t(\mathfrak L_n)$ with $\ad(\mathfrak L_n)$ is (linearly) generated by $\ad_{\overline \partial_n}$ (Lemma \ref{DKn_cap_Ln}). Now, $\ad_{\overline \partial_n}$ is the image by $\tau$ of $\overline{c_{\partial_n}}$, which is a braid automorphism, whence our conclusion.
\end{proof}

\begin{rmq}
This does not depend of an identification of $P_{n+1}^*$ with $\Inn(F_n)P_n$ : one can choose the more algebraic isomorphism of Lemma \ref{FnPn_and_braids}, or the more geometric (and arguably more interesting) isomorphism given by Magnus' theorem (Th~\ref{Pn*_acts_faithfully}). The latter allows us to reformulate our result in more geometric terms : the Johnson kernels associated with the action of $P_n^*$ on the fundamental group of the $n$-punctured sphere are exactly the terms of the lower central series of $P_n^*$. In other words, the Milnor invariants associated with the Magnus action of $B_{n+1}^*$ on $F_n$ distinguish elements of $B_{n+1}^*$ up to elements of the lower central series of $P_{n+1}^*$, exactly like in the classical case (see Remark \ref{rmq_on_Milnor_inv}).
\end{rmq}

\clearpage

\appendix

\section{Appendix: Braid and braid-like derivations}\label{section_braid_derivations}

Recall (from \cref{reminders_DK}) that the Artin action of $P_n$ on $F_n$ induces an action of $\Gamma_*(P_n)$ on $\Gamma_*(F_n)$ and thus an action of the Drinfeld-Kohno Lie ring $\Lie(P_n) = DK_n$ on the free Lie ring $\Lie(F_n) = \mathfrak L_n$. Moreover, the Johnson morphism encoding the latter is in fact an injection:
\[\tau: DK_n \hookrightarrow \Der(\mathfrak L_n).\]
We identify $DK_n$ with its image in $\Der(\mathfrak L_n)$, and its elements are called \emph{braid derivations}. It is easy to see that this image is contained in the Lie subring $\Der_t^{\partial}(\mathfrak L_n)$ of \emph{braid-like derivations} (see definition \ref{defi_braid_der}). The goal of the present section is to compare these two Lie subrings of $\Der(\mathfrak L_n)$. 

Most of the results in this appendix are not knew, and are even well-known in the theory of Milnor invariants, with a different point of view ($\Der_t^{\partial}(\mathfrak L_n)$ corresponding to Milnor invariants of string links, whereas $DK_n$ corresponds to Milnor invariants of braids). The only notable exception is Proposition \ref{rank_of_difference}, which estimates the ranks of the graded abelian group $\Der_t^{\partial}(\mathfrak L_n)/DK_n$.

\subsection{A calculation of ranks}

As graded abelian groups, both $DK_n$ and $\Der_t^{\partial}(\mathfrak L_n)$  are fairly well understood. Since the free Lie ring $\mathfrak L_n$ does not have torsion, neither do they. Moreover, we can compute explicitly their ranks in each degree $k$, from the ranks of the free Lie algebra, denoted by $d(n,k):= \rk_k(\mathfrak L_n)$, which are in turn given by Witt's formula \cite[I.4]{Serre}:
\begin{equation}\label{Witt_formula}
d(n,k) = \frac1k \sum\limits_{st = k} \mu(s)n^t,
\end{equation}
where $\mu$ is the usual M\"obius function.

\medskip

In order to compute the ranks of $DK_n$, consider the decomposition:
\[DK_n \cong \mathfrak L_{n-1} \rtimes \left(\mathfrak L_{n-2} \rtimes \left( \cdots \rtimes \mathfrak L_1\right) \cdots \right).\]
As an immediate consequence, we get:
\begin{equation}\label{rk(DK)}
rk_k(DK_n) = \sum\limits_{l=1}^{n-1} d(l,k).
\end{equation}

We now compute the ranks of $\Der_t^{\partial}(\mathfrak L_n)$. We first recall that given any choice of $n$ elements $\delta_i$ of $\mathfrak L_n$, there exists a unique derivation $\delta$ of $\mathfrak L_n$ sending each generator $X_i$ to $\delta_i$. Moreover, for each $X_i$, the only elements $Y$ of $\mathfrak L_n$ such that $[Y,X_i] =0$ are integral multiples of $X_i$ (see Lemma \ref{centralizers_in_Ln} below). As a consequence, the map sending $(t_i) \in (\mathfrak L_n)^n$ to the derivation $\delta: X_i \mapsto [X_i,t_i]$ is well-defined, and its kernel is $\mathbb Z \cdot X_1 \times \cdots \times \mathbb Z \cdot X_n \cong \mathbb Z^n$ (concentrated in degree $1$). Moreover, by definition of tangential derivations, its image is $\Der_t(\mathfrak L_n)$, whence:
\[\rk_k \left(\Der_t(\mathfrak L_n) \right) =
\begin{cases}
n(n-1)        &\ \text{ if }\ k=1, \\
n \cdot d(n,k)&\ \text{ if }\ k \geq 2.
\end{cases}\]
Now, $\Der_t^{\partial}(\mathfrak L_n)$ is the kernel of the linear map (raising the degree by $1$) $ev_\partial: \delta \mapsto \delta(\overline \partial_n) = \delta(X_1) + \cdots + \delta(X_n)$ from $\Der_t(\mathfrak L_n)$ to $\mathfrak L_n$. This map is surjective onto $(\mathfrak L_n)_{\geq 2}$. Indeed, $\mathfrak L_n$ is generated by so-called \emph{linear monomials} (this is easy to prove -- see for instance the appendix of \cite{Darne4}), which are of the form $[X_i,t]$ ($t \in \mathfrak L_n$), and $[X_i,t]$ is the image  by $ev_\partial$ of the derivation $\delta$  sending $X_i$ to $[X_i,t]$ and all other $X_j$ to $0$. From this, we recover the classical formula (compare~\cite[Th.~15]{Orr}):
\begin{align}\label{rk(braid_der)}
\begin{split}
\rk_k \left(\Der_t^\partial(\mathfrak L_n) \right) 
&= \rk_k \left(\Der_t(\mathfrak L_n) \right) - \rk_{k+1} (\mathfrak L_n) \\
&= \begin{cases}
n(n-1) - \frac{n(n-1)}2 = \frac{n(n-1)}2 &\ \text{ if }\ k=1, \\
n \cdot d(n,k) - d(n,k+1) &\ \text{ if }\ k \geq 2.
\end{cases} 
\end{split}
\end{align}

In order to understand the image of $DK_n$ by the Johnson morphism $DK_n \hookrightarrow \Der_t^\partial(\mathfrak L_n)$, we now compare formulas \eqref{rk(DK)} and \eqref{rk(braid_der)}. 

\subsection{Comparison in degree \texorpdfstring{$1$}{1}}

In degree one, \eqref{rk(DK)} and \eqref{rk(braid_der)} give:
\[rk_1(DK_n) = \sum\limits_{k=1}^{n-1} k = \frac{n(n-1)}2 = \rk_1 \left(\Der_t^\partial(\mathfrak L_n) \right).\]
In fact, the Johnson morphism in degree one is given by the explicit morphism:
\[\tau_1: t_{ij} \longmapsto \tau(t_{ij}): X_l \mapsto 
\begin{cases}
[X_i,X_j] &\ \text{ if }\ l=i, \\
[X_j,X_i] &\ \text{ if }\ l=j, \\
0 &\ \text{ else.}
\end{cases}\]
This formula is easily deduced from the fact that the action of $DK_n$ on $\mathfrak L_n$ is encoded in $\mathfrak L_n \rtimes DK_n \cong DK_{n+1}$ and from the following relations in $DK_{n+1}$ (see Prop.~\ref{Drinfeld-Kohno}):
\[[t_{ij}, t_{l, n+1}] =
\begin{cases}
[t_{i, n+1},t_{j, n+1}] &\ \text{ if }\ l=i, \\
[t_{j, n+1},t_{i, n+1}] &\ \text{ if }\ l=j, \\
0 &\ \text{ else.}
\end{cases}\]
The morphism $\tau_1$, which is known to be injective, is easily seen to be surjective. Indeed, if $\delta \in \Der_t^\partial(\mathfrak L_n)$ is homogeneous of degree $1$, then the equality  $\delta(X_1 + \cdots + X_n) = 0$ implies that the coefficient $\lambda_{ij}$ of $[X_i, X_j]$ in $\delta(X_i)$ equals the coefficient of $[X_j, X_i]$ in~$\delta(X_j)$, from which we deduce that $\delta = \sum \lambda_{ij} \tau(t_{ij}) = \tau\left(\sum \lambda_{ij} t_{ij} \right)$, the sum being taken on $(i,j)$ with~$i < j$.

\medskip

Since $DK_n$ is generated in degree one (Prop.~\ref{engdeg1}), we deduce:

\begin{prop}\label{DK_eng_deg_1}
The Drinfeld-Kohno Lie ring $DK_n$ identifies, \emph{via} the Johnson morphism, with the Lie subring of $\Der_t^\partial(\mathfrak L_n)$ generated in degree one.
\end{prop}

\subsection{Comparison in degree \texorpdfstring{$2$}{2}}

In degree $2$, using Witt's fomula \eqref{Witt_formula} for $k = 2$ and $k = 3$ (which give $d(n,2) = n(n-1)/2$ and $d(n,3) = (n^3 - n)/3$), we find:
\[\rk_2(DK_n) = \frac{n(n-1)(n-2)}6 = \rk_2 \left(\Der_t^\partial(\mathfrak L_n) \right).\]

We can in fact write an explicit formula for $\tau_2$, and show directly that it is an isomorphism:
\[\tau_2: [t_{ik}, t_{jk}] \longmapsto \delta_{ijk}:= [\tau(t_{ik}), \tau(t_{jk})]: X_l \mapsto 
\begin{cases}
[X_i,[X_j, X_k]] &\ \text{ if }\ l=i, \\
[X_j,[X_k, X_i]] &\ \text{ if }\ l=j, \\
[X_k,[X_i, X_j]] &\ \text{ if }\ l=k, \\
0 &\ \text{ else.}
\end{cases}\]
Remark that the fact that $\delta_{ijk}(X_1 + \cdots + X_n) = 0$ is exactly the Jacobi identity for $X_i$, $X_j$ and $X_k$.

\begin{prop}
The Johnson morphism $\tau: DK_n \rightarrow \Der_t^\partial(\mathfrak L_n)$ is an isomorphism in degree $2$.
\end{prop}

\begin{proof}
We already know that $\tau_2$ is injective : we only need to show that it is surjective. Let $\delta \in \Der_t^\partial(\mathfrak L_n)$ be homogeneous of degree $2$. Recall that $\mathfrak L_n$ is $\mathbb N^n$-graded, since the antisymmetry and Jacobi relation only relate parenthesized monomials with the same letters. We denote by $\deg: \mathfrak L_n \rightarrow \mathbb N^n$ the corresponding degree, sending $X_i$ to~$e_i$. Since $\delta$ is tangential, for all $i$, $\delta(X_i)$ can be written as a sum of $\lambda_{ij}[X_i,[X_i,X_j]]$ and of $\lambda_{ijk}[X_i,[X_j,X_k]]$ (with $j < k$), for some integral coefficients $\lambda_{ij}, \lambda_{ijk}$. Then,  in the decomposition of $\delta(X_1) + \cdots + \delta(X_n)$ into $\mathbb N^n$-homogeneous components, the components are:
\[\begin{cases}
\lambda_{ij}[X_i,[X_i,X_j]] &\text{ for }\ i \neq j, \\
\lambda_{ijk}[X_i,[X_j,X_k]] + \lambda_{jik}[X_j,[X_i,X_k]] + \lambda_{kij}[X_k,[X_i,X_j]] &\text{ for }\ i<j<k, \\
\end{cases} \]
respectively of degree $2e_i + e_j$ and $e_i + e_j + e_k$. The first ones are trivial if and only if $\lambda_{ij} = 0$ for all $i,j$. The second ones are trivial if and only if they are multiples of Jacobi identities, that is, if $\lambda_{ijk} = -\lambda_{jik} = \lambda_{kij}$. Indeed, the component of degree $e_i + e_j + e_k$ of $\mathfrak L_n$ is the quotient of the free abelian group generated by $[X_i,[X_j,X_k]]$, $[X_j,[X_i,X_k]]$ and $[X_k,[X_i,X_j]]$ by the Jacobi identity. When these conditions hold, we conclude that $\delta = \sum \lambda_{ijk} \delta_{ijk} \in \ima(\tau_2)$.
\end{proof}

\subsection{Comparison in higher degree}

In order to compute ranks in degree $3$, we use Witt's formula in degree $4$, which gives $d(n,4) = (n^4 - n^2)/4$. We find:
\begin{align*}
&\rk_3(DK_n) = \frac{(n-2)(n-1)n(n+1)}{12}, \\[0.5em]
&\rk_3 \left(\Der_t^\partial(\mathfrak L_n) \right) = \frac{(n-1)n^2 (n+1)}{12} .
\end{align*}
Hence the following:

\begin{prop}
The inclusion $DK_n \subset \Der_t^\partial(\mathfrak L_n)$ is a strict one: not all braid-like derivations come from braids.
\end{prop}

We can in fact say more:

\begin{prop}\label{rank_of_difference}
The cokernel of the inclusion of graded abelian groups $DK_n \subset \Der_t^\partial(\mathfrak L_n)$ is a graded abelian group whose rank in degree $k \geq 3$ is given by a polynomial function of $n$, whose leading term is $\frac{n^k}{2k}$.
\end{prop}

\begin{proof}

Recall that in degree $k$, the rank of $\Der_t^\partial(\mathfrak L_n)$ is $n \cdot d(n,k) - d(n, k+1)$ (formula~\eqref{rk(braid_der)}). Witt's formula \eqref{Witt_formula} implies that $d(n,k)$ is a polynomial function of~$n$, whose leading term is $n^k/k$ (since $\mu(1) = 1$). Moreover, the second non-trivial term is in degree $k/p$, where $p$ is the least prime factor of $k$, thus its degree is at most $k/2$. As a consequence, $n \cdot d(n,k) - d(n, k+1)$ is a polynomial function of $n$, with leading term
\[n \cdot \frac{n^k}k - \frac{n^{k+1}}{k+1} = \frac{n^{k+1}}{k(k+1)},\]
and no term of degree $k$.

Now, recall that in degree $k$, the rank of $DK_n$ is  given by formula \eqref{rk(DK)}:
\[\rk_k(DK_n) = \sum\limits_{l=1}^{n-1} d(l,k) 
= \sum\limits_{l=1}^{n-1} \frac 1k (l^k - l^{k/p} + \cdots) 
= \frac 1k \left( \sum\limits_{l=1}^{n-1} l^k - \sum\limits_{l=1}^{n-1} l^{k/p} + \cdots \right).\]
From Faulhaber's formula for sums of powers (recalled below), we know that $S_\alpha(n)$ is a polynomial function of $n$, whose leading terms are given by:
\[S_\alpha(n) = \frac{n^{\alpha + 1}}{\alpha + 1} + \frac{n^\alpha}2 + \cdots\]
Thus, the first two terms of $\rk_k(DK_n)$ are those of $\frac 1k S_k(n-1)$, that is:
\begin{align*}
\frac 1k S_k(n-1) 
&= \frac 1k \left(\frac{(n-1)^{k+1}}{k+1} + \frac{(n-1)^k}2 + \cdots\right) \\
&= \frac{n^{k+1}}{k(k+1)} - \frac {n^k}{2k} + \cdots
\end{align*}
We find, as announced, that the difference $\rk_k(\Der_t^\partial(\mathfrak L_n)) - \rk_k(DK_n)$ is a polynomial function of $n$, whose leading term is $\frac{n^k}{2k}$.
\end{proof}

We have used in the proof the usual formula (known as Faulhaber's formula) for sums of powers. For any integer $\alpha \geq 0$:
\[S_\alpha(n):= \sum\limits_{l = 0}^n l^\alpha = \frac 1{\alpha + 1}\sum\limits_{j=0}^\alpha \binom{\alpha + 1}{j}B_j \cdot n^{\alpha + 1 - j},\]
where the $B_j$ are the Bernoulli numbers, defined by:
\[\frac {ze^z}{e^z - 1} = \sum\limits_{j = 0}^\infty B_j \frac {z^j}{j!}.\]
Recall that the above formula can be obtained by a quite straightforward calculation of the exponential generating series:
\[\sum\limits_{\alpha = 0}^\infty S_\alpha(n) \frac{z^\alpha}{\alpha !} = \sum\limits_{l = 1}^n e^{lz} = e^z \cdot \frac{e^{nz} - 1}{e^z - 1}.\]

\section{Appendix: Studying \texorpdfstring{$P_n^*$}{Pn*}}\label{section_action_Pn*}

We gather here some classical results about $P_n^* = P_n/\mathcal Z(P_n)$, with some new proofs. Our first goal is  a classical theorem of Magnus \cite{Magnus} about the faithfulness of its action on the fundamental group of the punctured sphere. We give a complete proof which avoids any difficult calculation with group presentations. This proof is partially built on the sketch of proof of \cite[Lem. 3.17.2]{Birman}. The proof uses the Hopf property in a crucial way, which leads us to recall some basic facts about Hopfian groups. Finally, we describe the splitting $P_n \cong P_n^* \times \mathcal Z(P_n)$, which can be used to give an alternative proof of part of the results of \cref{par_Inn(Pn)} (see in particular Remark~\ref{rmq_on_splitting}).

\subsection{A faithful action of \texorpdfstring{$P_n^*$}{Pn*} on \texorpdfstring{$F_{n-1}$}{the free group}}\label{par_action_Pn*}

Recall that the pure braid group $P_n$ acts faithfully on $F_n$ \emph{via} the Artin action. Since this action is by group automorphisms fixing $\partial_n = x_1 \cdots x_n$, there is an induced action of $P_n$ on the quotient $F_n/\partial_n \cong F_{n-1}$ (see \cref{par_Andreadakis_for_Pn*}). We now prove the classical:

\begin{theo}\label{Pn*_acts_faithfully} \textup{\cite[Formula (23)]{Magnus}.}
The Artin action induces a faithful action of $P_n^* = P_n/\mathcal Z(P_n)$ on  $F_n/\partial_n \cong F_{n-1}$. Moreover, this induces an isomorphism between $P_n^*$ and the subgroup $\Inn(F_{n-1})P_{n-1}$ of $\Aut(F_{n-1})$.
\end{theo}

\begin{rmq}
The restriction to pure braids is not an important one: we can deduce readily that the result holds for the action of $B_n^*$ on $F_{n-1}$, by showing that the kernel of this action is contained in $P_n$. Indeed, let $\beta \in B_n$ and $\sigma_\beta$ its image in $\Sigma_n$. If $\beta$ acts trivially on $F_n/\partial_n \cong F_{n-1}$, then it acts trivially on $F_n^{ab}/\overline{\partial_n}$. But the action of $B_n$ on $F_n^{ab} \cong \Z^n$ is \emph{via} the canonical action of $\Sigma_n$ on $\Z^n$, and $\Z^n/(X_1 + \cdots + X_n)$ is a faithful representation of $\Sigma_n$, hence $\sigma_\beta = 1$.
\end{rmq}

Our proof will use the following lemma, giving an isomorphism between $P_n^*$ and $\Inn(F_{n-1})P_{n-1}$, which is different from the one of the theorem:

\begin{lem}\label{FnPn_and_braids}
The kernel of the morphism $(w, \beta) \mapsto c_w\beta$ from $F_n \rtimes P_n \cong P_{n+1}$ to $\Inn(F_n)P_n \subset \Aut(F_n)$ is exactly the center of $P_{n+1}$. Thus, $\Inn(F_n)P_n \cong P_{n+1}^*$.
\end{lem}

\begin{proof}
The formula does defines a morphism, because of the usual formula $\beta c_w \beta^{-1} = c_{\beta(w)}$. The center of $P_{n+1}$ is generated by $(\partial_{n+1}, c_{\partial_{n+1}}^{-1})$ (Prop.~\ref{Z(Pn)}), which is obviously sent to the identity. Conversely, if $(w, \beta)$ is such that $c_w = \beta^{-1}$, then $c_w \in P_n \cap \Inn(F_n)$. Lemma \ref{Pn_cap_Fn} implies that $w$ must be a power of $\partial_{n+1}$, whence the result.
\end{proof}

\begin{proof}[Proof of Theorem \ref{Pn*_acts_faithfully}]

Let us denote by $\pi: P_n \twoheadrightarrow P_n^*$ the canonical projection and by $\chi: P_n^* \rightarrow \Aut(F_{n-1})$ the action under scrutiny.

\smallskip

\textbf{We first show that $\ima(\chi) = \Inn(F_{n-1})P_{n-1}$.}
From the definition of the action, the characterization of braid automorphisms of $F_n$ and the description of $\Inn(F_{n-1})P_{n-1}$ given in Lemma \ref{description_FnPn}, it follows directly that $\chi$ takes values in $F_{n-1}P_{n-1}$. The map $s: P_{n-1} \hookrightarrow P_n$ extending braid automorphisms by $x_n \mapsto x_n$ is easily seen to satisfy $\chi \pi s (\beta) = \beta$ for all $\beta \in P_{n-1}$, hence $P_{n-1} \subset \ima(\chi)$. Moreover, for all $j \leq n-1$, the inner automorphism $c_{x_1 \cdots x_j}$ is the image by $\chi$ of the braid automorphism:
\[C_j: x_t \longmapsto
\begin{cases}
{}^{(x_1 \cdots x_j)}\! x_t &\text{ if }\ t \leq j,\\
{}^{(x_{j+1} \cdots x_n)^{-1}}\! x_t &\text{ if }\ t > j.
\end{cases}\]
Since the $x_1 \cdots x_j$ form a basis of $F_{n-1}$, we conclude that $\Inn(F_{n-1}) \subset \ima(\chi)$.

\smallskip

\textbf{We now show that the action $\chi$ is faithful.}
Lemma \ref{FnPn_and_braids} implies that $\Inn(F_{n-1})P_{n-1}$ is isomorphic to $P_n^*$. Thus $\chi$ restricts to a morphism from  $P_n^*$ to $P_n^*$. We have just showed that this endomorphism is surjective. We claim that its surjectivity implies that it is an automorphism of $P_n^*$. Indeed, since $P_n$ embeds in $IA_n$ (\emph{via} the Artin action), it is residually nilpotent. Hence, by Lemma \ref{G/Z_res_nilp}, $P_n^*$ is too. Thus, Proposition \ref{res_nilp_are_Hopfian} implies that it is \emph{Hopfian}, which means exactly that surjective endomorphisms of $P_n^*$ are automorphisms.
\end{proof}

\begin{rmq}
The braid automorphism $C_j$ from the proof of Theorem~\ref{Pn*_acts_faithfully} is easy to understand as a geometric braid. Here is a drawing of this braid:
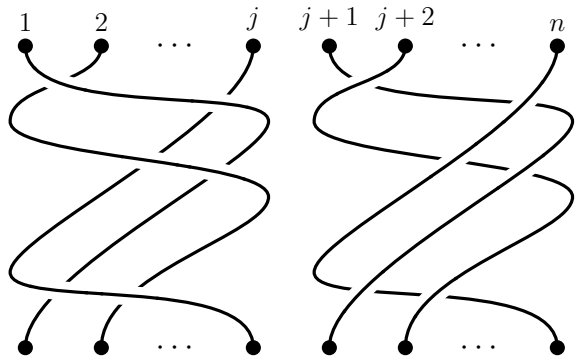
\[\begin{tikzpicture}
\node[draw=none] at (2,4) {$\cdots$};
\node[draw=none] at (2,0) {$\cdots$};
\node[draw=none] at (6,4) {$\cdots$};
\node[draw=none] at (6,0) {$\cdots$};
\begin{knot}[clip width = 6, flip crossing/.list={1,3,5,8,10,12}]
\strand[very thick]  
   (0, 0)  .. controls  +(0, 1) and +(0, -0.5) .. (3.2, 3)
   node[at start,circle, fill, inner sep=2pt]{}
   .. controls  +(0, 0.5) and +(0, -1) .. (0, 4)
   node[at end,circle, fill, inner sep=2pt]{}
   node[at end, above=2, scale=0.85]{$1$};
\strand[very thick]  
   (1, 0)  .. controls  +(0, 1) and +(0, -0.5) .. (3.2, 2)
   node[at start,circle, fill, inner sep=2pt]{} 
   .. controls  +(0, 0.5) and +(0, -0.5) .. (-.2, 3) 
   .. controls  +(0, 0.5) and +(0, -0.5) .. (1, 4)
   node[at end,circle, fill, inner sep=2pt]{}
   node[at end, above=2, scale=0.85]{$2$};
\strand[very thick]  
   (3, 0)  .. controls  +(0, 1) and +(0, -0.5) .. (-.2, 1)
   node[at start,circle, fill, inner sep=2pt]{}
   .. controls  +(0, 0.5) and +(0, -1) .. (3, 4)
   node[at end,circle, fill, inner sep=2pt]{}
   node[at end, above=2, scale=0.85]{$j$};
\strand[very thick]  
   (4, 0)  .. controls  +(0, 1) and +(0, -0.5) .. (7.2, 3)
   node[at start,circle, fill, inner sep=2pt]{}
   .. controls  +(0, 0.5) and +(0, -1) .. (4, 4)
   node[at end,circle, fill, inner sep=2pt]{}
   node[at end, above=2, scale=0.85]{$j+1$};
\strand[very thick]  
   (5, 0)  .. controls  +(0, 1) and +(0, -0.5) .. (7.2, 2)
   node[at start,circle, fill, inner sep=2pt]{} 
   .. controls  +(0, 0.5) and +(0, -0.5) .. (3.8, 3) 
   .. controls  +(0, 0.5) and +(0, -0.5) .. (5, 4)
   node[at end,circle, fill, inner sep=2pt]{}
   node[at end, above=2, scale=0.85]{$j+2$};
\strand[very thick]  
   (7, 0)  .. controls  +(0, 1) and +(0, -0.5) .. (3.8, 1)
   node[at start,circle, fill, inner sep=2pt]{}
   .. controls  +(0, 0.5) and +(0, -1) .. (7, 4)
   node[at end,circle, fill, inner sep=2pt]{}
   node[at end, above=2, scale=0.85]{$n$};
\end{knot}
\end{tikzpicture}\]

Moreover, if we interpret the braid group as the mapping class group of the punctured disc \cite[Th.~1.10]{Birman}, then $C_j$ identifies with the (commutative) product of Dehn twists along $\gamma_1$ and $\gamma_2$ (left-handed along $\gamma_1$ and right-handed along $\gamma_2$):
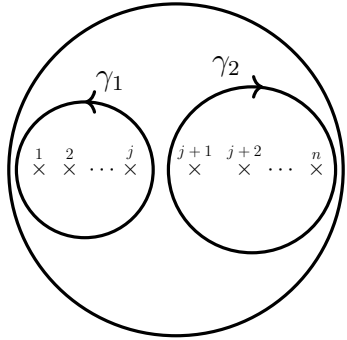
\[\begin{tikzpicture}
\draw[very thick] (0,0) circle (2.2) ;
\draw[very thick] (-1.2,0) circle (0.9) ;
\draw (-1.2,0.9) node [above] {$\gamma_1$}; 
\draw[very thick] (1,0) circle (1.1) ;

\draw (1,1.1) node [above] {$\gamma_2$}; 
\draw (-1.8,0) node [scale=0.7] {$\times$}; 
\draw (-1.8,0) node [above=2, scale=0.5] {$1$}; 
\draw (-1.4,0) node [scale=0.7] {$\times$}; 
\draw (-1.4,0) node [above=2, scale=0.5] {$2$}; 
\draw (-0.95,0)node [scale=0.7]{$\cdots$}; 
\draw (-0.6,0) node [scale=0.7] {$\times$}; 
\draw (-0.6,0) node [above=2, scale=0.5] {$j$}; 
\draw (0.25,0) node [scale=0.7] {$\times$}; 
\draw (0.25,0) node [above=2, scale=0.5] {$j+1$}; 
\draw (0.9,0)  node [scale=0.7] {$\times$}; 
\draw (0.9,0)  node [above=2, scale=0.5] {$j+2$}; 
\draw (1.4,0)  node [scale=0.7] {$\cdots$}; 
\draw (1.85,0) node [scale=0.7] {$\times$}; 
\draw (1.85,0) node [above=2, scale=0.5] {$n$}; 
\end{tikzpicture}\]

\end{rmq}

\begin{rmq} Instead of using the Hopf property, one could define explicitly the inverse of $\chi: P_n^* \rightarrow P_n^*$ using the explicit lifts of the generators given in the proof. However, showing that this inverse is well-defined directly from the presentation of $P_n^*$ does involve quite a bit of calculation, which one would need to do in a clever way in order not to get lost. This method would be closer to the original proof of Magnus \cite[Formula (23)]{Magnus}. Our method is closer to the sketch of proof of \cite[Lem. 3.17.2]{Birman}, but the latter seems to miss the fact that the endomorphism of $P_n^*$ at the end of our proof is a non-trivial one, whose injectivity is not obvious. 
\end{rmq}

\begin{rmq}[Braids on the cylinder]
The action of $P_n^*$ on $F_n$ described algebraically in the lemma also has a geometric interpretation. In fact, it extends to an action of $F_n \rtimes B_n$, which is the group $B_n(\D - pt)$ of braids with $n$ strands on the punctured disk (or on the cylinder), on $F_n$, which is the fundamental group of the disk minus $n$ points. Precisely, let $C_n(X)$ denote the unordered configuration space of $n$ points on~$X$, and $\Homeo_*(\mathbb D)$ denote the group of self-homeomorphisms of $\mathbb D$ fixing the basepoint~$pt$. The choice of a base configuration $\mathbf c = (c_i)$ on the punctured disc $\mathbb D - pt$ induces an evaluation map $\Homeo_*(\mathbb D) \rightarrow C_n(\D - pt)$. The latter is a locally trivial fibration whose fiber is the subgroup $\Homeo_*(\mathbb D, \mathbf c)$ of homeomorphisms permuting the $c_i$. Then, using the well-known fact that $SO_2(\mathbb R)$ is a retract of $\Homeo_*(\mathbb D)$, and the fact that $\pi_2(C_n(\D - pt)) = 0$ \cite[Prop.~1.3]{Birman}
the long exact sequence in homotopy gives the short exact sequence: 
$0 \rightarrow \Z \rightarrow B_n(\D - pt) \rightarrow \pi_0\Homeo_*(\mathbb D, \mathbf c) \rightarrow 0$.
The mapping class group $\pi_0\Homeo_*(\mathbb D, \mathbf c)$ acts transparently on the fundamental group of the disk minus $n$ points, which is $F_n$. It is not difficult to see that the corresponding action of $B_n(\D - pt) \cong F_n \rtimes B_n$ is the one from the lemma, and that the generator of $\pi_1(SO_2) = \Z$ is sent to the central element of $B_n(\D - pt)$. This action was also considered for instance in \cite{Bodin}, and with this interpretation, our Lemma \ref{FnPn_and_braids} is equivalent to their Theorem~2.1.
\end{rmq}

\subsection{Residually nilpotent groups}

Recall that a group $G$ is called \emph{Hopfian} if every surjective endomorphism of $G$ is an automorphism. The relevance of this property in our context relies on the following:

\begin{prop}\label{res_nilp_are_Hopfian}
Residually nilpotent groups of finite type are Hopfian.
\end{prop}

\begin{proof}
Let $G$ be a residually nilpotent groups of finite type and $\pi: G \rightarrow G$ be a surjective endomorphism. The Lie ring $\Lie (G)$ is generated in degree one. This implies on the one hand, that every $\Lie_k(G)$ is finitely generated, and on the other hand, that $\Lie(\pi)$ is surjective ($\Lie_1(\pi)$ is, being the induced endomorphism of $G^{ab}$). Because of Lemma \ref{abelian_are_Hopfian}, this implies that all the $\Lie_k(\pi)$ are isomorphisms.

Suppose now that there exists $x \in \ker(\pi)$ such that $x \neq 1$. Since $G$ is residually nilpotent, there exists $k \geq 1$ such that $x \in \Gamma_kG - \Gamma_{k+1}G$. Then $\overline x$ is a non-trivial element of $\ker(\Lie_k(\pi))$, which is impossible.
\end{proof}

\begin{lem}\label{abelian_are_Hopfian}
Abelian groups of finite type are Hopfian.
\end{lem}

\begin{proof}
Let $A$ be an abelian group of finite type and $\pi: A \rightarrow A$ be a surjective endomorphism. If $A$ has no torsion, then the rank of $\ker(\pi)$ must be trivial, hence $\ker(\pi) = 0$, whence the result in this case.

\smallskip
In general, let $\Tors(A)$ denotes the torsion subgroup of $A$. Then $\pi$ induces a commutative diagram whose rows are short exact sequences:
\[\begin{tikzcd}
\Tors(A) \ar[r, hook] \ar[d, "\pi_\#"] &A \ar[r, two heads] \ar[d, two heads, "\pi"] & A/\Tors(A) \ar[d, "\overline \pi"] \\
\Tors(A)  \ar[r, hook] &A \ar[r, two heads]  & A/\Tors(A). 
\end{tikzcd}\] 
Since $\pi$ is surjective, the induced endomorphism $\overline \pi$ of $A/\Tors(A)$ has to be too. But $A/\Tors(A)$ is a free abelian group of finite type. As a consequence of the first part of the proof, $\overline \pi$ has to be an isomorphism. Then, the snake lemma implies that the induced endomorphism $\pi_\#$ of $\Tors(A)$ is surjective. But $\Tors(A)$ is finite, so $\pi_\#$ must be an isomorphism too. We can then apply the snake lemma again to conclude that $\pi$ is injective.
\end{proof}

\begin{lem}\label{G/Z_res_nilp}
If $G$ is a residually nilpotent group, then so is $G/\mathcal Z(G)$.
\end{lem}

\begin{proof}
Let $x \in G$ such that $\bar x \in \Gamma_n \left(G/\mathcal Z(G)\right)$. Then $x \in \mathcal Z(G) \Gamma_n (G)$, whence $[G,x] \subseteq \Gamma_{n+1} (G)$. Thus, if $\bar x \in \bigcap \Gamma_n \left(G/\mathcal Z(G)\right)$, then $[G,x] \subseteq \bigcap \Gamma_{n+1} (G) = \{1\}$, which means that $x \in \mathcal Z(G)$, and $\bar x = 1$.
\end{proof}

\subsection{Splitting by the center}

We now show that there is a (non-canonical) splitting $P_n \cong P_n^* \times \mathcal Z(P_n)$ (Cor.~\ref{splitting_of_center_for_Pn}), which we replace in the following general context:

\begin{prop}\label{splitting_of_center}
For a group $G$, the following conditions are equivalent:
\begin{enumerate}[align=left, itemsep=-3pt,topsep=3pt, labelwidth=19pt, label=(\roman*), ref=\textbf{(\roman*)}]
\item \label{item-1} $G/\mathcal Z G$ is centerless, and there exists an isomorphism $G \cong \mathcal Z G \times (G/\mathcal Z G)$.
\item \label{item-2} The canonical projection $p: G \twoheadrightarrow G/\mathcal Z G$ splits.
\item \label{item-3} The canonical map $\mathcal Z G \hookrightarrow G \twoheadrightarrow G^{ab}$ is injective and its image is a direct factor of $G^{ab}$.
\end{enumerate}
When $G$ is of finite type, or Hopfian, these are equivalent to:
\begin{enumerate}[itemsep=-3pt,topsep=3pt, label=(\roman*'), ref=(\roman*')]
\item \label{item-1'} There exists an isomorphism $G \cong \mathcal Z G \times (G/\mathcal Z G)$.
\end{enumerate}
\end{prop}

\begin{proof}
\ref{item-3} $\boldsymbol{\Rightarrow}$ \ref{item-2}: Suppose that the third condition is satisfied. We denote by $\pi: G \twoheadrightarrow G^{ab}$ the canonical projection. Let us identify $\mathcal Z G$ with its image in $G^{ab}$, and let $W$ be a direct complement of $\mathcal Z G$ in $G^{ab}$. Then we claim that $G$ decomposes as the direct product of its subgroups $\mathcal Z G$ and $\pi^{-1}(W)$. Indeed, if $g \in G$, then its class $\pi(g)$ decomposes as $z + w$ with $z \in \mathcal Z G$ and $w \in W$, and $gz^{-1} \in \pi^{-1}(W)$, which implies that $g \in \mathcal Z G \cdot \pi^{-1}(W)$. Moreover, the intersection of the two subgroups is trivial by definition of $W$, and elements of $\mathcal Z G$ commute with elements of $\pi^{-1}(W)$, whence our claim. Then the canonical projection $p: G \twoheadrightarrow G/\mathcal Z G$ has to induce an isomorphism $\pi^{-1}(W) \cong G/\mathcal Z G$, whose inverse is a section of $p$.

\smallskip

\ref{item-2} $\boldsymbol{\Rightarrow}$ \ref{item-1}: If $p$ has a section $s$, then $G$ decomposes as a semi-direct product $Z G \rtimes (G/\mathcal Z G)$, which has to be a direct product, since the conjugation action of $G$ on $\mathcal Z G$ is trivial. Moreover, if $z \in \mathcal Z(G/\mathcal Z G)$, then $s(z)$ is central in $\mathcal Z G \times (G/\mathcal Z G) \cong G$, thus $z =ps(z)$ is trivial, by definition of $p$.

\smallskip

\ref{item-1} $\boldsymbol{\Rightarrow}$ \ref{item-3}:
Now suppose that $G/\mathcal Z G$ is centerless. Then $\mathcal Z(\mathcal Z G \times (G/\mathcal Z G)) = \mathcal Z G \times 1$. Recall that the canonical map $\mathcal Z G \hookrightarrow G \twoheadrightarrow G^{ab}$ is functorial in $G$. As a consequence, an isomorphism  between $G$ and $\mathcal Z G \times (G/\mathcal Z G)$ has to induce a commutative square:
\[\begin{tikzcd}
\mathcal Z G \ar[d, "\cong"] \ar[r]
&G^{ab} \ar[d, "\cong"] \\
\mathcal Z G \times 1 \ar[r, hook]
& \mathcal Z G \times (G/\mathcal Z G)^{ab},
\end{tikzcd}\]
whence the third condition.

\smallskip

\textbf{If $G$ is Hopfian,} then any direct factor of $G$ is Hopfian: if $G \cong H \times K$ and $u$ is a surjective endomorphism of $H$, then $u \times 1$ is a surjective endomorphism of $H \times K \cong G$, thus is an automorphism, and $u$ must be injective too. Thus, if $G \cong \mathcal Z G \times (G/\mathcal Z G)$, then $\mathcal Z G$ is Hopfian. But we also have an induced isomorphism $\mathcal Z G \cong \mathcal Z G \times  \mathcal Z(G/\mathcal Z G)$. If we compose this isomorphism with the first projection, we get a surjective endomorphism of $\mathcal Z G$, which has to be injective, forcing $\mathcal Z(G/\mathcal Z G)$ to be trivial, whence the result.

\smallskip

\textbf{If $G$ is of finite type} and $G \cong \mathcal Z G \times (G/\mathcal Z G)$, then $\mathcal Z G$ is a quotient of $G$, thus it is abelian of finite type, hence Hopfian, and the same reasoning as in the case when $G$ is Hopfian leads to the desired conclusion.
\end{proof}

\begin{cor}\label{splitting_of_center_for_Pn}
There exists an isomorphism $P_n \cong \mathcal Z (P_n) \times P_n^*$.
\end{cor}

\begin{proof}
The center $\mathcal Z (P_n)$ injects into $P_n^{ab} \cong \Z \{t_{ij}\}_{i<j}$, and its image is $\Z \cdot \sum_{i<j} t_{ij}$, which is a direct factor of $P_n^{ab}$.
\end{proof}

\begin{rmq}
Their is no canonical choice of splitting of $P_n \twoheadrightarrow P_n^*$: the splitting depends on a choice of direct complement $W$ of $\mathcal Z (P_n)$ inside $P_n^{ab}$. For instance, we can choose $W = W_{kl}$, generated by all the $t_{ij}$ for $(i,j) \neq (k,l)$, so that the corresponding section $s_{kl}$ sends the class of $A_{ij}$ to $A_{ij}$ if $(i,j) \neq (k,l)$. In the literature, authors often choose $s_{1,2}$. Another natural choice is the section corresponding to $W = \{ \sum \lambda_{ij}t_{ij} | \sum \lambda_{ij} = 0 \}$, which sends the class of $A_{ij}A_{kl}^{-1}$ to $A_{ij}A_{kl}^{-1}$, for all $i,j,k,l$.
\end{rmq}

\bibliographystyle{alpha}
\bibliography{Ref_Inner_Aut}

\begin{thebibliography}{CPVW08}

\bibitem[And65]{Andreadakis}
Stylianos Andreadakis.
\newblock On the automorphisms of free groups and free nilpotent groups.
\newblock {\em Proc. London Math. Soc. (3)}, 15:239--268, 1965.

\bibitem[Art25]{Artin25}
Emil Artin.
\newblock Theorie der {Z}\"{o}pfe.
\newblock {\em Abh. Math. Sem. Univ. Hamburg}, 4(1):47--72, 1925.

\bibitem[Art47]{Artin47}
Emil Artin.
\newblock Theory of braids.
\newblock {\em Ann. of Math. (2)}, 48:101--126, 1947.

\bibitem[AT08]{Alekseev}
Anton Alekseev and Charles Torossian.
\newblock The {K}ashiwara-{V}ergne conjecture and {D}rinfeld’s associators.
\newblock {\em Annals of Mathematics}, 175:415--463, 2008.

\bibitem[Bar16]{Bartholdi1}
Laurent Bartholdi.
\newblock Automorphisms of free groups. {I}---erratum [ {MR}3084710].
\newblock {\em New York J. Math.}, 22:1135--1137, 2016.

\bibitem[BB16]{Bodin}
Paolo Bellingeri and Arnaud Bodin.
\newblock The braid group of a necklace.
\newblock {\em Math. Z.}, 283(3-4):995--1010, 2016.

\bibitem[Bir74]{Birman}
Joan~S. Birman.
\newblock {\em Braids, links, and mapping class groups}.
\newblock Princeton University Press, Princeton, N.J.; University of Tokyo
  Press, Tokyo, 1974.
\newblock Annals of Mathematics Studies, No. 82.

\bibitem[BN16]{Bardakov-Neshchadim}
Valeriy~G. Bardakov and Mikhail~V. Neshchadim.
\newblock Subgroups, automorphisms, and {L}ie algebras related to the
  basis-conjugating automorphism group.
\newblock {\em Algebra Logika}, 55(6):670--703, 2016.

\bibitem[Cho48]{Chow}
Wei-Liang Chow.
\newblock On the algebraical braid group.
\newblock {\em Ann. of Math. (2)}, 49:654--658, 1948.

\bibitem[CPVW08]{CPVW}
Frederick~R. Cohen, Jonathan~N. Pakianathan, Vladimir~V. Vershinin, and J.~Wu.
\newblock Basis-conjugating automorphisms of a free group and associated {L}ie
  algebras.
\newblock In {\em Groups, homotopy and configuration spaces}, volume~13 of {\em
  Geom. Topol. Monogr.}, pages 147--168. Geom. Topol. Publ., Coventry, 2008.

\bibitem[Dar19a]{Darne4}
Jacques Darn\'e.
\newblock Milnor invariants of braids and welded braids up to homotopy.
\newblock arXiv:1904.10677, 2019.

\bibitem[Dar19b]{Darne2}
Jacques Darn\'e.
\newblock On the {A}ndreadakis conjecture for subgroups of {$IA_n$}.
\newblock {\em International Mathematics Research Notices}, 2019.

\bibitem[Dar19c]{Darne1}
Jacques Darn\'{e}.
\newblock On the stable {A}ndreadakis problem.
\newblock {\em J. Pure Appl. Algebra}, 223(12):5484--5525, 2019.

\bibitem[FR87]{Falk-Randell}
Michael Falk and Richard Randell.
\newblock The lower central series of generalized pure braid groups.
\newblock In {\em Geometry and topology ({A}thens, {G}a., 1985)}, volume 105 of
  {\em Lecture Notes in Pure and Appl. Math.}, pages 103--108. Dekker, New
  York, 1987.

\bibitem[HM00]{Habegger}
Nathan Habegger and Gregor Masbaum.
\newblock The {K}ontsevich integral and {M}ilnor's invariants.
\newblock {\em Topology}, 39(6):1253--1289, 2000.

\bibitem[Ibr19]{Ibrahim}
Abdoulrahim Ibrahim.
\newblock On a subgroup of mccool group {$P\Sigma_n$}.
\newblock arXiv:1902.10033, 2019.

\bibitem[Iha89]{Ihara}
Yasutaka Ihara.
\newblock The {G}alois representation arising from {$\mathbb P^1 - \{0,1,
  \infty\}$} and {T}ate twists of even degree.
\newblock In {\em Galois groups over {$\mathbb Q$}}, 1989.

\bibitem[Kal50]{Kaloujnine2}
L\'eo Kaloujnine.
\newblock Sur quelques propri\'et\'es des groupes d'automorphismes d'un groupe
  abstrait. ({G}\'en\'eralisation d'un theor\`eme de {M}. {P}h. {H}all).
\newblock {\em C. R. Acad. Sci. Paris}, 231:400--402, 1950.

\bibitem[Koh85]{Kohno}
Toshitake Kohno.
\newblock S\'erie de {P}oincar\'e-{K}oszul associ\'ee aux groupes de tresses
  pures.
\newblock {\em Invent. Math.}, 82(1):57--75, 1985.

\bibitem[Laz54]{Lazard}
Michel Lazard.
\newblock Sur les groupes nilpotents et les anneaux de {L}ie.
\newblock {\em Ann. Sci. Ecole Norm. Sup. (3)}, 71:101--190, 1954.

\bibitem[Mag34]{Magnus}
Wilhelm Magnus.
\newblock \"{U}ber {A}utomorphismen von {F}undamentalgruppen berandeter
  {F}l\"{a}chen.
\newblock {\em Math. Ann.}, 109(1):617--646, 1934.

\bibitem[MW02]{Mostovoy}
Jacob Mostovoy and Simon Willerton.
\newblock Free groups and finite-type invariants of pure braids.
\newblock {\em Math. Proc. Cambridge Philos. Soc.}, 132(1):117--130, 2002.

\bibitem[Orr89]{Orr}
Kent~E. Orr.
\newblock Homotopy invariants of links.
\newblock {\em Invent. Math.}, 95(2):379--394, 1989.

\bibitem[Sat17]{Satoh_triangulaire}
Takao Satoh.
\newblock On the {A}ndreadakis conjecture restricted to the
  ``lower-triangular'' automorphism groups of free groups.
\newblock {\em J. Algebra Appl.}, 16(5):1750099, 31, 2017.

\bibitem[Ser06]{Serre}
Jean-Pierre Serre.
\newblock {\em Lie algebras and {L}ie groups}, volume 1500 of {\em Lecture
  Notes in Mathematics}.
\newblock Springer-Verlag, Berlin, 2006.
\newblock 1964 lectures given at Harvard University, Corrected fifth printing
  of the second (1992) edition.

\end{thebibliography}

\end{document}